\documentclass[11pt]{amsart}
\usepackage{multicol}
\usepackage{amssymb}
\usepackage{amsmath}
\usepackage{amsthm}
\usepackage{thmtools}
\usepackage{amsrefs}
\usepackage{mathrsfs}
\usepackage{faktor}
\usepackage{tikz}
\usepackage{mathtools}

\newcounter{bobcomments}

\usetikzlibrary{arrows}
\usetikzlibrary{shapes.geometric}
\usetikzlibrary{decorations.pathreplacing,decorations.markings}
\usepackage[utf8]{inputenc}
\usepackage[top=1in,bottom=1in,right=1in,left=1in]{geometry}
\usepackage{hyperref}
\usepackage{enumitem}
\usepackage{wasysym}

\usepackage[capitalise]{cleveref}

\newtheorem{theorem}{Theorem}[section]
\newtheorem{lem}[theorem]{Lemma}
\newtheorem{proposition}[theorem]{Proposition}
\newtheorem{cor}[theorem]{Corollary}
\newtheorem{qu}[theorem]{Question}

\theoremstyle{definition}
\newtheorem{dfn}[theorem]{Definition}
\newtheorem{ex}[theorem]{Example}
\newtheorem{rmk}[theorem]{Remark}
\newtheorem{ntn}[theorem]{Notation}

\numberwithin{theorem}{section}

\DeclareMathOperator{\Aut}{Aut}

\newcommand{\Z}{\mathbb{Z}}

\title[Equations in Seifert 3-manifold groups]{
Decidability of equations and first-order theory \\ in  Seifert 3-manifold groups }

\thanks{The research of the authors was supported by the EPSRC Fellowship grant EP/V032003/1 ‘Algorithmic, topological and geometric aspects of infinite groups, monoids and inverse semigroups’. }

\begin{document}
\maketitle

\vspace{-4mm}

\begin{center}
ROBERT D. GRAY
\footnote{School of Engineering, Mathematics, and Physics, University of East Anglia, Norwich NR4 7TJ, England.
Email \texttt{Robert.D.Gray@uea.ac.uk}.
}
and
ALEX LEVINE
\footnote{School of Engineering, Mathematics, and Physics, University of East Anglia, Norwich NR4 7TJ, England.
Email \texttt{A.Levine@uea.ac.uk}.
}

\end{center}

\begin{abstract}
In   
\cite{Aschenbrenner2015}*{Question 5.2 \& Question 5.3}
Aschenbrenner, Friedl and Wilton
ask: 
(1) Is the equation problem solvable for the fundamental group of any $3$-manifold? and 
(2) Is the first-order theory of the fundamental group of any $3$-manifold  decidable? 
In this paper we answer both of these questions by proving that Hilbert's tenth problem over the integers can be encoded in equations over any non-virtually abelian fundamental group of any Seifert fibered 3-manifold whose orbifold has non-negative Euler characteristic.
We use this to show that the equation problem (and hence also the first-order theory) is undecidable in this infinite family of $3$-manifold groups and then apply it to classify the Seifert 3-manifold groups with decidable equation problems and decidable first-order theories, in the case that the orbifold has non-negative Euler characteristic. 
In contrast, we show that for this class of Seifert 3-manifold groups the single equation problem is decidable. 
For every Seifert 3-manifold group $G$ 
where the orbifold has negative Euler characteristic
we show that either $G$ has decidable equation problem or $G$ has a finite index subgroup of index $2$ that has decidable equation problem. 
These negative Euler characteristic results follow from work of Liang on central extensions of hyperbolic groups. 
We also discuss why Liang's results do not suffice to deal with all the negative Euler characteristic cases.     
We show how to construct several other infinite families of $3$-manifold groups with undecidable equation problem (and hence also undecidable first-order theory) including examples that are not Seifert manifold groups and examples that are not virtually nilpotent. 
In addition, we observe that there are numerous other infinite families for which the first-order theory is undecidable such as  
fundamental groups of manifolds modeled on 3-dimensional Sol geometry.
\end{abstract}

\section{Introduction}\label{sec:intro}
Fundamental groups of compact $3$-manifolds enjoy many good algorithmic
properties e.g. they all have decidable word problem, conjugacy problem and
subgroup membership problem; see \cite{Hempel1987}, \cite{Pr06}, \cite{Pr16},
\cite{Friedl2016}, and \cite{Aschenbrenner2015}*{Section~4}. Furthermore, the
isomorphism problem is decidable for the case of closed, orientable
3–manifolds, and the homeomorphism problem is decidable for orientable,
irreducible 3–manifolds; see \cite{Sela1995} and
\cite{Aschenbrenner2015}*{Section~4}. These results are discussed in detail in
the survey article of Aschenbrenner, Friedl and Wilton \cite{Aschenbrenner2015}
on decision problems for $3$-manifolds and their fundamental groups. At the end
of that paper 
the authors state several open problems and conjectures.
Specifically they ask: 

\begin{qu}\label{AFWOpen1}\cite{Aschenbrenner2015}*{Question 5.2} Is the equation problem solvable for the fundamental group of any 3–manifold? \end{qu}

\vspace{-4mm}

Here the \emph{equation problem}, also called the \emph{Diophantine problem}, which generalises both the word problem and the conjugacy problem, asks for a solution to the problem whether any finite set of ``equations'' over a group has a solution; see below for a formal definition. More generally, they pose the following question which asks about decidability of the full elementary theory of 3–manifold groups viewed as a structures in the language of groups in the model-theoretic sense.

\begin{qu}\label{AFWOpen1}\cite{Aschenbrenner2015}*{Question 5.3} Let $\pi$ be a $3$-manifold group. Is the first-order theory of $\pi$ decidable? \end{qu}

\vspace{-4mm}

The main purpose of this paper is to answer both of these questions by proving that Hilbert's tenth problem over the integers can be encoded in equations over any non-virtually abelian fundamental group of any Seifert fibered 3-manifold whose orbifold has non-negative Euler characteristic. This allows us to give an explicit classification of the Seifert fibered 3-manifolds with decidable equation problems, and those with decidable first-order theories, in this case. For the other case of Seifert 3-manifold groups where the orbifold has negative Euler characteristic we shall explain how results of Liang \cite{Liang2014} can be applied to give an up to index $2$ classification of those Seifert 3-manifold groups with decidable equation problems. 
We shall also explain and clarify the reasons why Liang's results cannot immediately be applied to deal with all the negative characteristic cases, and how this relates to the open problem of whether decidability the equation problem passes to finite index extensions.

In more detail, our first main result is the following (see Section~\ref{sec:prelim} below for the formal definitions of the terms used in the statement). 
Note that in the statement of the following result, and throughout the rest of this paper, we shall opt to use the expression \emph{Diophantine problem} rather than \emph{equation problem} since it has become the more commonly used term used for this decision problem in the literature.  

\vspace{1.5mm}

\noindent \textbf{Theorem~A.} 
\emph{
Let $M$ be a Seifert fibered manifold with base orbifold $B$ and set $G=\pi_1(M)$. 
\begin{enumerate} 
\item If $\chi(B) \geq 0$ then either 
\begin{enumerate} 
\item $G$ is not virtually abelian in which case Hilbert's tenth problem over $\Z$  is reducible to the Diophantine problem in $G$. In particular $G$ has an undecidable Diophantine problem and thus also has an undecidable first-order theory; or else
\item $G$ is virtually abelian and thus has a decidable first-order theory with ETD0L solutions. 
In particular $G$ has a decidable Diophantine problem with EDT0L solutions.  
  \end{enumerate}
\item If $\chi(B) < 0$ then either $G$ has a decidable Diophantine problem or $G$ contains an index $2$ subgroup that has a decidable Diophantine problem.     
  \end{enumerate}
} 

\vspace{1.5mm}

The full statement of Theorem~A is given in Theorem~\ref{Seifert-DP} below where we precisely characterise the cases where these problems are and are not decidable in terms of the parameters in the finite presentations defining the group. In particular it can be read off from Theorem~\ref{Seifert-DP} that there are several infinite families of examples of Seifert $3$-manifold groups for which the Diophantine problem is undecidable.      

Part (2) of Theorem~A is obtained as a corollary of a result of Liang \cite{Liang2014}. 
We conjecture that in case (2) of Theorem~A the group $G$ will always have decidable Diophantine problem. 
We note that in Liang's original paper \cite{Liang2014} he pointed out that his main result in that paper could be applied to certain Seifert $3$-manifold groups. Indeed Liang's result concerns central extensions of hyperbolic groups and as Liang points out in his paper \cite{Liang2014} the fundamental groups of a large and interesting class of Seifert fibered spaces are of this form. Part (2) of Theorem~A and Theorem~\ref{Seifert-DP} clarify the exact extent to which Liang's result on central extensions may be applied to prove decidability of the Diophantine problem for certain Seifert $3$-manifold groups.  
In particular, we see that Liang's result is not enough on its own to deal with all the negative characteristic examples.

When we restrict to single equations we shall show.  

\vspace{1.5mm}

\noindent \textbf{Theorem~B.} 
\emph{
Let $M$ be a Seifert fibered manifold with base orbifold $B$.
\begin{enumerate} 
\item If $\chi(B) \geq  0$ then the Seifert manifold group $G=\pi_1(M)$ has a decidable single equation problem. 
\item If $\chi(B) <  0$ then either $\pi_1(M)$ has a decidable single equation problem or it has an index 2 subgroup that does.   
  \end{enumerate}
} 

\vspace{1.5mm}

\vspace{1.5mm}

\noindent \textbf{Corollary~C.} 
\emph{
Every Seifert 3-manifold group $G$ has a subgroup $H \leq G$ with index at most $2$ such that $H$ has a decidable single equation problem.     
} 

\vspace{1.5mm}

We conjecture that all Seifert 3-manifold groups have decidable single equation problem. This also raises the interesting question of whether all $3$-manifold groups might have single equation problem.  

Our main results concern fundamental groups of Seifert fibered manifolds. These manifolds and their fundamental groups are an important class that are key building blocks in general theory of 3-manifold groups. Seifert fibered manifolds are important to the geometrisation theorem because they represent six of the eight Thurston geometries and serve as fundamental building blocks for many 3-manifolds.         
There are a range of algebraic results about residual properties of Seifert $3$-manifold groups that relate to algorithmic properties. For instance in \cite{Scott1978} 
Scott 
showed that Seifert 3-manifold groups are subgroup separable and 
in
\cite{Niblo1992} Niblo generalised that result by proving that they are double coset separable. 
Related work includes   
\cite{Sun2020}
which gives a 
characterisation of which finitely generated subgroups of finitely generated 3-manifold groups are separable. 
Using topological results, in  
\cite{Martino2007} Martino
showed that Seifert groups are conjugacy separable.  
 Allenby et al. give an algebraic proof of the same result in \cite{Allenby2005} and \cite{Allenby2010}.
We note that not all 3-manifold groups are subgroup separable see e.g.
\cite{Burns1987}, 
while in \cite{Hamilton2013} it is shown that 
the fundamental group of any  compact, orientable 3-manifold is conjugacy separable.

For the rest of the introduction we shall explain how the main results of this paper, described above, relate to other existing literature on the Diophantine problem, and first-order theory, of finitely generated groups.  
The study of the Diophantine problem in groups has a long and distinguished
history. In fundamental work Makanin \cites{Makanin_semigroups, Makanin1982}
proved that the Diophantine problem is decidable in free monoids and in free
groups. Building on this, Razborov \cite{Razborov_thesis} developed a powerful
method for describing sets of solutions to systems of equations in free groups
using what are now called Makanin-Razaborov diagrams. This work laid the
foundations that ultimately led to solutions to the Tarski problems
\cites{Kharlampovich2006, Sela2006} on the first-order theory of free groups.
The Diophantine problem has since been shown to be decidable for right-angled
Artin groups \cite{Diekert2001}, hyperbolic groups \cites{Sela, Dahmani2010}
and virtually direct products of hyperbolic groups \cite{CiobanuHoltRees}.
Torsion-free relatively hyperbolic groups have also been studied, with various
examples, such as those with virtually abelian parabolic subgroups having
decidable Diophantine problem \cite{dahmani2009existential}.

When a group has a decidable Diophantine problem it is of interest to study the properties of sets of solutions to systems of equations over that group. A recent focus has been on describing the solution set to a system of equations over groups using formal language theory via, so called, EDT0L languages. In \cite{Ciobanu2016} it was proved that the solution set of a word equation is an EDT0L language and in \cite{EDT0L_RAAGs} it was shown that the set of all solutions to a system of equations over a right-angled Artin group is an EDT0L language.   
Since then several other classes of groups have been shown to have the property that solutions to equations over the group can be represented by EDT0L languages, including hyperbolic groups \cite{eqns_hyp_grps} and virtually abelian groups \cite{VAEP}.

In the case that the orbifold has non-negative characteristic, our results on
the Diophantine problem for $3$-manifold groups connects with results in the
literature on solving equations over the Heisenberg group and other nilpotent
groups. Roman'kov \cite{Romankov1979} gave the first example of a nilpotent
group with undecidable Diophantine problem. Later Truss \cite{Truss1995} proved
that the free nilpotent group $N(3,2)$ of step $3$ and rank $2$ has undecidable
Diophantine problem. Some non-nilpotent undecidability results include free 
metabelian groups \cite{RomankovMetabelian}, the
restricted wreath product \(\Z \wr \Z\) \cite{Dong2025}, some other wreath
products of abelian groups \cite{KharlampovichMyasnikov25} and Thompson's group
\(F\) \cite{ElliottLevine}. The connection with the work in this paper comes
from the fact that the Seifert $3$-manifold groups with undecidable Diophantine
problem described in this paper include in particular the Heisenberg group
$H(\Z)$ which is isomorphic to the free nilpotent group $N(2, 2)$. The
Heisenberg group was shown to have undecidable Diophantine problem in
\cite{duchin_liang_shapiro} as part of a more general result showing that all
free nilpotent groups have undecidable Diophantine problem.  In contrast, in
the same paper \cite{duchin_liang_shapiro} they prove that the single equation
problem is decidable in the Heisenberg group.  Seifert 
$3$-manifold groups are typically not nilpotent groups, but if $M$ is a
Seifert fibered manifold with base orbifold $B$ such that $\chi(B)=0$, then
$\pi_1(M)$ is a virtually nilpotent group; in fact it is virtually class
$2$ nilpotent  (see Proposition~\ref{euclidean-Seifert-virt-nilp} below for a
proof of this).       

We obtain the main undecidability results in this paper by proving that there
are explicit encodings of Hilbert's tenth problem over the integers into every
non-virtually abelian fundamental group of any Seifert fibered 3-manifold whose
orbifold has non-negative Euler characteristic. In contrast, our positive
solution to the single equation problem arises from the fact (see
Proposition~\ref{euclidean-Seifert-virt-nilp} below) that all such groups are
virtually nilpotent. The result \cite{Garreta2020}*{Corollary~4.29} shows that
if $G$ is a finitely generated virtually nilpotent group that is not virtually
abelian then there exists a ring of algebraic integers $O$ e-interpretable in
$G$, and the  Diophantine problem in $O$ is reducible to the Diophantine
problem in $G$. Two very recent preprints \cite{koymans2024hilbert} and
\cite{alpoge2025rank} state results which would imply that these
generalisations of Hilbert's problem are also undecidable. Thus one consequence
of the results claimed in the two preprints \cites{koymans2024hilbert,
alpoge2025rank} combined with \cite{Garreta2020}*{Corollary~4.29} is that it
would give an alternative way of obtaining the corollary to Theorem~A stating
that the Diophantine problem is undecidable in non-virtually abelian Seifert
manifold groups in the case of non-negative orbifold characteristic. Note
however that that approach would not suffice to recover Theorem~A itself nor
recover the full statement of that result in Theorem~\ref{Seifert-DP}.

Taken together our results give an, up to index 2, classification of Seifert 3-manifold groups with decidable Diophantine problems. Note that it is an open question whether in general decidability of the Diophantine problem is preserved under finite index extensions or subgroups. Even for virtual RAAGs this question remains open.
The results in this paper are a step towards the ultimate goal of classifying the $3$-manifold groups with decidable Diophantine problems and those with decidable first-order theories. At the end of the paper in Section~\ref{sec:other} below we discuss what we currently know regarding these questions for some other $3$-manifold groups. As explained e.g. in \cite[Theorem 1.8.1]{AschenbrennerBook2015} Seifert fibered spaces account for all compact oriented manifolds in six of the eight Thurston geometries. It is natural to ask what we can say for the fundamental groups for the other two geometries. Some properties of these groups are summarised in \cite[Table~1.1]{AschenbrennerBook2015}. Of the remaining two geometries, a natural family worth considering is fundamental  groups of manifolds modeled on 3-dimensional Sol geometry. We do not know whether such groups have decidable Diophantine problem. However, as we shall explain in Section~\ref{sec:other}, they do have an undecidable first-order theory.    
In the same section we shall also show how to construct several other infinite families of $3$-manifold groups with undecidable Diophantine problem including examples that are not Seifert manifold groups and examples that are not virtually nilpotent.

\section{Preliminaries}
\label{sec:prelim}

\subsection{Notation}
We recall some standard concepts and clarify certain notations that we will use
throughout.

We use $FG(A)$ to denote the free group on alphabet $A$. If \(G\) is a group
and \(x, y \in G\) we define \([x, y] = xyx^{-1}y^{-1}\). The
\emph{centraliser} \(C_G(x)\) of an element \(x\) in a group \(G\) is defined
by \(C_G(x) = \{y \in G : xy = yx\}\).  If $G$ is a group, $N$ is a normal
subgroup of $G$ and $S$ is any subset of $G$ then we define \( S / N = \{ Nx : x
\in S \} \); that is, the set of all cosets of $N$ that $S$ intersects. Note
that there is no assumption that $N$ is contained in the set $S$ here.
When working with a group $G$ defined by a presentation we write $u \equiv v$
to say $u$ and $v$ are equal as words, and $u=v$ to mean they represent the
same element of the group. If multiple groups are present, we may write $u
=_G v$ to clarify that we mean they are equal in the group $G$.

\subsection{Equations and first-order theory of groups}

	We now define equations and the Diophantine problem in a group.

	\begin{dfn}
			Let \(G\) be a finitely generated group and \(\mathcal X\) be a
			finite set. An \emph{equation} in \(G\) is an element \(w \in G
			\ast FG(\mathcal X)\) and denoted \(w = 1\). A \emph{solution} to
			\(w = 1\) is a homomorphism \(\phi \colon G \ast FG(\mathcal X) \to
			G\) which fixes \(G\) pointwise, such that \((w) \phi = 1\). We
			refer to the elements of \(\mathcal X\) as the \emph{variables} of
			the equation. A \emph{(finite) system of equations} in \(G\) is a
			finite set of equations in \(G\) using the same set of variables. A
			\emph{solution} to a system \(\mathcal E\) of equations is a
			homomorphism that is a solution to every equation in \(\mathcal
			E\). If \(\mathcal E\) is a system, when writing out \(\mathcal E\),
			we usually use the logical operator \(\wedge\) to delimit equations
			in the system; that is \(\mathcal E = (w_1 = 1) \wedge (w_2 = 1)
			\wedge \cdots \wedge (w_n = 1)\).

	The \emph{Diophantine problem} in \(G\) is the decision question asking if
	there is an algorithm taking a system \(\mathcal E\) of equations in \(G\)
	as input and outputting whether or not \(\mathcal E\) admits a solution. The
	\emph{single equation problem} in \(G\) is the decision question asking if
	there is an algorithm that takes a single equation \(w = 1\) in \(G\)
	and outputs if it admits a solution.
\end{dfn}
    
		We will often abuse notation, and consider a solution to an equation
		with \(n \in \mathbb{Z}_{>0}\) variables \(X_1, \ldots, X_n\) in a
		group \(G\) to be an \(n\)-tuple of elements \((g_1, \ \ldots, \
		g_n)\), instead of a homomorphism from \(G \ast FG(X_1, \ \ldots, \
		X_n) \to G\). One can recover a homomorphism \(\phi\) from a 
		tuple by setting \(\phi\) to be the unique homomorphism satisfying 
		\((g) \phi = g\) for all \(g \in G\) and \((X_i) \phi = g_i\) for
		all \(i\). We
		further abuse notation by considering equations of the form
		\(w_1 = w_2\). Such an equation is used to denote the equation 
		\(w_1w_2^{-1} = 1\).

	We now define various technical terms relating to equations that will be
	frequently used throughout our proofs.
	\begin{dfn}
		Let \(G\) be a finitely generated group. If \(\mathcal E\) is a
		system of equations in \(G\) and \((X_1, \ldots, X_n)\) are
		a collection of (not necessarily all) variables occurring in
		\(\mathcal E\), we define the set of solutions to
		\((X_1, \ldots, X_n)\) in \(\mathcal E\) to be set
		\[
		\{(g_1, \ldots, g_n) : \textrm{ there exist } h_1, \ldots, h_m
				\in G \textrm{ such that } (g_1, \ldots, g_n, h_1, \ldots, h_m)
\text{ is a solution to } \mathcal E\}.\]

		Let \(S \subseteq G^n\), where \(n\) is a
		positive integer. We say \(S\) is \emph{equationally definable} in
		\(G\) if there is a system \(\mathcal E\) of equations in \(G\) with
		\(m \geq n\) variables, where there are \(n\) variables
		\((X_1, \ldots, X_n)\) in \(\mathcal E\), such that the set of
		solutions to \((X_1, \ldots, X_n)\) in \(\mathcal E\) is \(S\).
	\end{dfn}
We can combine equationally definable sets to make new equationally definable sets. For instance, if a subset $S \subseteq G$ of a group $G$ is equationally definable then the set of squares $X = \{ x^2 : x \in S \}$ is clearly also equationally definable. Note that the set of squares $X$ is not the same thing as the square $S^2$ of the set $S$.        

		We provide some straightforward examples of equations
		in an abelian group.

		\begin{ex}
			Let \(G = \langle a, b, c \mid ab = ba, ac = ca, bc = cb, c^4 = 1
			\rangle \cong \Z^2 \times C_4\). Consider the equation \(X^2 = 1\)
			in \(G\). It is not difficult to see that the only element of order
			\(2\) in \(G\) is \(c^2\), and thus the set of solutions to \(X^2 =
			1\), where \(X\) is a variable, is simply \(\{c^2\}\). Similarly,
			the system of equations \((Y = Xa) \wedge (Z = X^4)\) has the
			set of solutions \(\{(a^i b^j c^k, a^{i + 1} b^j c^k, a^{4i} b^{4j})
			: i, j \in \Z, k \in \{0, 1, 2, 3\}\}\).
		\end{ex}

	\begin{dfn}
		Let \(\mathcal Q_1\) and \(\mathcal Q_2\) be decision problems (for
		example, the Diophantine problem in a group). We say \(\mathcal Q_1\)
		is \emph{reducible} to \(\mathcal Q_2\) if an there is an algorithm
		that takes as input an instance \(I_1\) of \(\mathcal Q_1\) and outputs
		an instance \(I_2\) of \(\mathcal Q_2\) such that \(I_1\) is true if
		and only if \(I_2\) is true.
	\end{dfn}

	An important use of reducing one decision problem to another is the fact
	that if \(\mathcal Q_1\) reduces to \(\mathcal Q_2\) and \(Q_2\) is
	decidable, then so is \(\mathcal Q_1\). In our case, we frequently use
	the contrapositive of this, by showing that the undecidable Hilbert's
	tenth problem over the ring of integers is reducible to the Diophantine
	problem in certain groups.

	One important generalisation of the Diophantine problem in a group is the
	first-order theory (with constants). Roughly speaking, a \emph{first-order
	sentence (with constants)} in a group is any finite expression that can be
	constructed using variables, constants from the group, any of the logical
	operators \(\forall\), \(\exists\), \(\wedge\), \(\vee\), \(\neg\) and
	\(=\), together with the group's multiplication \(\cdot\).  The
	\emph{first-order theory} in a group (as a decision problem) asks if there
	is an algorithm that decides if a given first-order sentence admits a solution. We refer to \cite{hodges1993model, marker2002model} for a formal introduction to first-order theories of groups.

\subsection{Hilbert's tenth problem}
 As alluded to above, our source of undecidability will be Hilbert's tenth
 problem over the ring of integers. This decision problem takes as input a
 finite system of integer polynomials over the same (finite) set of variables,
 and asks if this system admits a solutions in integers. This was shown to be
 undecidable by Matijasevi\v{c} \cite{Matijasevic1970}, and has since been
 applied to show that the Diophantine problem is undecidable in various
 different groups.

\subsection{Equations in groups and EDT0L languages}

	A recent development in the study of equations in groups is the use of
	EDT0L languages to describe sets of solutions. None of the proofs we give
	use this definition directly - we simply cite existing literature to
	conclude that certain groups have solutions that can be described using
	EDT0L languages.  We thus omit the definition, and refer the reader to a
	survey on the use of EDT0L languages to describe equations
	\cite{CiobanuLevine} and a textbook giving a more detailed study into EDT0L
	languages as an abstract class \cite{math_theory_L_systems}. We say a group
	has a \emph{Diophantine problem with EDT0L solutions} if the set of
	solutions can be described using an EDT0L language, as defined formally in
	\cite{CiobanuLevine}.  Ciobanu and Evetts have also used EDT0L languages to
	describe solutions to first-order sentences in virtually abelian groups
	\cite{CiobanuEvetts}, and so we analogously define the term
	\emph{first-order theory with EDT0L solutions}.

\begin{rmk}\label{rmk:virtuallyAbelianEDT0L} 
When we say a group has \emph{decidable Diophantine problem with EDT0L solutions} this means that the Diophantine problem is decidable and the set of solutions to any finite system of equations has a description as an EDT0L language. Note that this does not necessarily mean that there is an algorithm that takes a finite system of equations over the group and returns an EDT0L language describing the solution set. The same comment applies when we talk about a group having \emph{decidable first-order theory with ETD0L solutions}.    
Combining the results in the papers  
\cite{CiobanuEvetts, Ersov1972} 
it follows that virtually abelian groups have decidable first-order theory with ETD0L solutions.
 \end{rmk}

\subsection{Seifert 3-manifold groups}
Seifert fibered spaces are a fundamental class of 3-manifolds that were
originally introduced by Seifert in 1933; see \cite{Seifert1933}. The definition was
later extended to include singular fibres which reverse orientation and it is
the more general definition that we work with here; see
\cite[Section~3]{Scott1983}. A \emph{Seifert fibered space} is a 3-manifold $M$ with a
decomposition of $M$ into disjoint circles, that are called \emph{fibres}, such
that each circle has a neighbourhood in $M$ which is a union of fibres and is
isomorphic to a fibered solid torus of Klein bottle. Here the \emph{trivial
fibred solid torus} is the product $S^1 \times D^2$ with the product foliation
by circles $S^1 \times \{ y \}, y \in D^2$, and a \emph{fibred solid torus} is
a solid torus with a foliation by circles which is finitely covered by a
trivial fibered solid torus. In a similar way, a \emph{fibred solid Klein
bottle} is a solid Klein bottle that is  finitely covered by a trivial fibered
solid torus.  Given any Seifert fibered space $M$, one can obtain an orbifold
$B$, called the \emph{base orbifold}, by quotienting out by the $S^1$-action on
the fibres of $M$ i.e. by identifying each fiber to a point.

Our interest is in Seifert 3-manifold groups which are the fundamental groups of Seifert fibered spaces. We begin by recalling the classification of Seifert 3-manifold groups and their presentations.

\begin{theorem}[see, for example {\cite{OrlikNotes}*{Section 5.2}}]
				\label{Seifert-pres}
Let \(M\) be a Seifert fibered manifold with  
				base orbifold \(B\).
				Then:
				\begin{enumerate}
								\item If \(B\) is orientable:
\[
 \pi_1(M) \cong {\displaystyle \langle u_{1},v_{1},...u_{g},v_{g},q_{1},...q_{r},h|u_{i}h=h^{\epsilon }u_{i},v_{i}h=h^{\epsilon }v_{i},q_{i}h=hq_{i},q_{j}^{\alpha_{j}}h^{\beta_{j}}=1,q_{1}...q_{r}[u_{1},v_{1}]...[u_{g},v_{g}]=h^{b}\rangle } 
\]
where \(\epsilon \in \{-1, 1\}\), \(\alpha_i, \beta_j \in \Z_{\geq 1}\) are
coprime for each \(j\), and \(b \in
\Z\), such that one of
the following two cases holds:
\begin{enumerate}
\item \emph{Type o1}: \(\epsilon = 1\) and \(0 < \beta_j < \alpha_j\) for all \(j\);
\item \emph{Type o2}: \(g > 0\), \(\epsilon = -1\), \(0 < \beta_j \leq \frac{\alpha_j}{2}\) for
				all \(j\) and \(b \in \{0, 1\}\), unless \(\alpha_j = 2\) for some \(j\),
				in which case \(b = 0\).
\end{enumerate}
\(\epsilon = -1\) being of \emph{type o2};
\item If \(B\) is non-orientable and \(g > 0\): 
\[
 \pi_1(M) \cong {\displaystyle \langle v_{1},...,v_{g},q_{1},...q_{r},h|v_{i}h=h^{\epsilon _{i}}v_{i},q_{i}h=hq_{i},q_{j}^{\alpha_{j}}h^{\beta_{j}}=1,q_{1}...q_{r}v_{1}^{2}...v_{g}^{2}=h^{b}\rangle }
\]
where each \(\epsilon_i \in \{-1, 1\}\) such that one of the following four cases holds:
\begin{enumerate}
\item \emph{Type n1}: \(\epsilon_i = 1\) for all \(i\), \(0 < \beta_j \leq
				\frac{\alpha_j}{2}\) for all \(j\) and \(b \in \{0, 1\}\), unless
				\(\alpha_j = 2\) for some \(j\), in which case \(b = 0\);
\item \emph{Type n2}: \(\epsilon_i = -1\) for all \(i\), and \(0 < \beta_j <
				\alpha_j\) for all \(j\);
\item \emph{Type n3}: \(g \geq 2\), \(\epsilon_1 = 1\) and \(\epsilon_i = -1\) for all \(i >
				1\), \(0 < \beta_j \leq \frac{\alpha_j}{2}\) for all \(j\) and \(b \in
				\{0, 1\}\), unless \(\alpha_j = 2\) for some \(j\), in which case \(b =
				0\);
\item \emph{Type n4}: \(g \geq 3\), \(\epsilon_1 = \epsilon_2 = 1\) and \(\epsilon_i = -1\)
				for all \(i > 2\), \(0 < \beta_j \leq \frac{\alpha_j}{2}\) for all
				\(j\) and \(b \in \{0, 1\}\), unless \(\alpha_j = 2\) for some \(j\),
				in which case \(b = 0\).
\end{enumerate}
\end{enumerate}
\end{theorem}

\begin{ntn}
	Let \(M\) be a Seifert fibered manifold with base orbifold
	\(B\). Using the notation from \cref{Seifert-pres}, define
	\(\chi(B)\) by
	\[
		\chi(B) =
		\begin{cases}
				2 - (2g + \sum_{i = 1}^r (1 - \frac{1}{\alpha_i}))
			& \text{if } B \text{ is orientable} \\
			2 - (g + \sum_{i = 1}^r (1 - \frac{1}{\alpha_i}))
			& \text{if } B \text{ is non-orientable}.
		\end{cases}
	\]
	Furthermore, we say \(M\) (or \(\pi_1(M)\)) is \((g; \alpha_1, \ldots, \alpha_r)\).
\end{ntn}

We use $\pi_1^{orb}(B)$ to denote the \emph{orbifold fundamental group of $B$}. It is important to note 
this is not the same as 
the fundamental group of the underlying topological manifold, that is, it is not the same as the fundamental group of $B$. The orbifold fundamental group 
$\pi_1^{orb}(B)$ 
is related to the fundamental group of $M$ via e.g. \cite[Lemma~3.1]{Joecken2021} which in the notation of the presentations given in the statement of Theorem~\ref{Seifert-pres} says that
$\pi_1^{orb}(B) \cong \pi_1(M) / \langle h \rangle$.     

Now following \cite[Subsection~5.1]{Joecken2021} let $M$ be a 
Seifert fibered manifold 
with base orbifold $B$. Then one of the following holds: 
\begin{itemize} 
\item $B$ is a bad orbifold; or 
\item $B$ is a good orbifold, modelled on either $S^2$, $\mathbb{H}^2$ or $\mathbb{E}^2$. 
  \end{itemize}
Here by definition an orbifold is called \emph{good} if it is the quotient of a manifold by an action of a discrete group of isometries, and an orbifold that is not good is called \emph{bad}. 
Then we have the following:  
\begin{itemize} 
\item If $B$ is either a bad orbifold or a good orbifold modelled on $S^2$ then $\pi_1(M)$ is virtually cyclic; see \cite[Proposition 5.1]{Joecken2021}.
If $\pi_1(M)$ is infinite virtually cyclic then $\langle h \rangle$ is infinite cyclic by \cite{BoyerRolfsenWiest}*{Proposition 4.1(1)} and so $\pi_1(M) / \langle h \rangle$ is finite. Otherwise, $\pi_1(M)$ is finite and so $\pi_1(M) / \langle h \rangle$ is finite.        
\item If $B$ is a good orbifold modelled on $\mathbb{H}^2$ then 
the orbifold fundamental group 
$\pi_1^{orb}(B)$  
is hyperbolic; see \cite[Proof of Proposition 5.2]{Joecken2021}. 
\item If $B$ is a good orbifold modelled on $\mathbb{E}^2$ then the 
orbifold fundamental group 
$\pi_1^{orb}(B)$ 
is virtually $\mathbb{Z} \times \mathbb{Z}$; see \cite[Section~4]{BoyerRolfsenWiest}.  
  \end{itemize}
Combining these observations together with \cite[Theorem~5.3(2)]{Scott1983} gives the following well-known result. 

\begin{theorem}
		\label{Seifert-Euler-qts}
	Let \(M\) be a Seifert fibered manifold with base
	orbifold \(B\). Then 
	\begin{enumerate}
			\item If \(\chi(B) < 0\), then \(\pi_1(M) / \langle h
					\rangle\) is hyperbolic;
			\item If \(\chi(B) = 0\), then \(\pi_1(M) / \langle h
					\rangle\) is virtually \(\Z^2\);
			\item If \(\chi(B) > 0\), then \(\pi_1(M) / \langle h
					\rangle\) is finite.
	\end{enumerate}
	Moreover, \(\chi(B) = 0\) in precisely the following cases:
\begin{multicols}{2}
	\begin{itemize}
		\item Type (o1), (0; 2, 2, 2, 2);
		\item Type (o1), (0; 3, 3, 3);
		\item Type (o1), (0; 2, 4, 4);
		\item Type (o1), (0; 2, 3, 6); 
		\item Type (o1-2), (1;);
		\item Type (n1-2), (1; 2, 2);
		\item Type (n1-3), (2;).
	\end{itemize}
  \end{multicols}
\end{theorem}

\begin{lem}
		\label{h-inf-order}
	Let \(M\) be a Seifert fibered manifold with base orbifold
	\(B\) such that \(\chi(B) = 0\). Then, using the notation
	of \cref{Seifert-pres}, the element \(h\) has infinite order in
	\(\pi_1(M)\).
\end{lem}

\begin{proof}
		By \cite{BoyerRolfsenWiest}*{Proposition 4.1(1)}, \(h\)
		has finite order if and only if \(\pi_1(M)\) is finite.
		But \cref{Seifert-Euler-qts} tells us that \(\pi_1(M)\)
		projects onto a virtually \(\Z^2\) group, and thus
		is infinite.
\end{proof}

\section{Seifert 3-manifold groups with zero orbifold Euler characteristic}
In this section we consider the Seifert 3-manifold groups $\pi_1(M)$ where $M$ be a Seifert fibered manifold whose base orbifold $B$ satisfies $\chi(B) = 0$. 
The results in this section give a complete classification of which of these groups have decidable Diophantine problem. It will turn out for this class that in all cases that the Diophantine problem is decidable, the first-order theory will also be decidable. 
The results of this section will then be applied in the next section where we prove the first main result of the article Theorem~\ref{Seifert-DP}.  
Throughout this section in all cases where we show the Diophantine problem is undecidable we do this by giving an encoding Hilbert's tenth problem over the integers.

The next two results provide a framework for reducing Hilbert's tenth problem
to the Diophantine problem in various
examples on Seifert 3-manifold groups with whose base orbifold has zero Euler
charactersitic.
\begin{proposition}
				\label{comm-undecidable}
	Let \(G\) be a group admitting a normal infinite cyclic subgroup
	\(\langle h \rangle\). Further, suppose that there exist \(x, y \in G\)
	satisfying all of the following:
	\begin{enumerate}
		\item There exists \(c \in \Z \setminus \{0\}\) such that for all \(i, j \in \Z\),
						\([x^i, y^j] = h^{cij}\);
		\item There exist equationally definable sets \(S_x\) and \(S_y\), such that
						\(S_x / \langle h \rangle = \langle x \rangle / \langle h \rangle\)
						and \(S_y / \langle h \rangle = \langle y \rangle / \langle h \rangle\);
		\item \(\langle h \rangle\) is equationally definable.
	\end{enumerate}
Then the ring of integers $\Z$ is interpretable in $G$ by finite systems of equations, and thus Hilbert's tenth problem over $\Z$  is reducible to the Diophantine problem in $G$. 
Hence \(G\) has an undecidable Diophantine problem.
\end{proposition}

\begin{proof}
We accomplish this by embedding the integers into
	\(G\) as \(\langle h \rangle\). Since \(\langle h \rangle\) is equationally 
	definable, we can use the Diophantine problem to add constraints of the
	form \(X \in \langle h \rangle\). To recover the Diophantine problem in the
	ring of integers, we must encode two types of equation: \(U + V = W\) and
	\(UV = W\), where \(U\), \(V\) and \(W\) can denote constants or variables.
	The additive case is straighforward; we simply encode it as
	\(XY = Z\), with the constraints that \(X, Y, Z \in \langle h \rangle\),
	where \(X\) is a constant if and only if \(U\) is, in which case
	\(X = h^U\), and the analogous statements hold for \(Y\) and \(V\), and
	\(Z\) and \(W\).

	The equation \(UV = W\) requires a little extra work. Since \(\langle h
	\rangle\) is normal, and infinite cyclic, elements of \(G\) act by conjugation on \(\langle
	h \rangle\) by either fixing \(h\), or mapping \(h\) to \(h^{-1}\). Thus any
	element of \(G\) that is a square will centralise \(h\). In particular, by
	replacing \(x\) with \(x^2\) and \(y\) with \(y^2\), we can assume without
	loss of generality that \(xhx^{-1} = h\) and \(yhy^{-1} = h\). Note that this
	does not affect the assumption that \([x^i, y^j] = h^{cij}\), modulo
	a change of the constant \(c\).  For the second
	assumption, this is also okay, since the set of squares of elements of an
	equationally definable set is also equationally definable.
Furthermore, if $\overline{S_x}$ is the set of squares of elements of $S_x$ then 
since $S_x / \langle h \rangle = \langle x \rangle / \langle h \rangle$ 
we have 
\[
\overline{S_x} / \langle h \rangle =
\{ g^2 \langle h \rangle : g \in S_x \}   
= \{ k^2 \langle h \rangle : k \in \langle x \rangle  \}   
= \langle x^2 \rangle / \langle h \rangle,   
\]   
and similarly for the corresponding set $\overline{S_y}$ of squares of elements of $S_y$.  

	Next note that since \(x\) and \(y\) both commute with \(h\), we have \([x^i
	h^{k_1}, y^j h^{k_2}] = h^{cij}\). We now show that the set \(R_x  = 
	\{(x^i h^k, h^i) : x^i h^k \in S_x\}\) is equationally definable.  We show
	that it is the set of solutions to the variables \((X, Z)\) within the equation \([X,
	y] = Z^c\), with the constraints \(X \in S_x\) and \(Z \in \langle h
	\rangle\). This follows since, by the above \([x^i h^k, y] = h^{ci}\), and so
	\((x^i h^k, h^l)\) is a solution if and only if 
$[x^i h^k, y] = h^{lc}$ which holds if and only if $h^{ci} = h^{lc}$ if and only if  
\(i = l\), since $c \neq 0$. 
Hence \(R_x\) is
	equationally definable. By symmetry \(R_y  = \{(y^j h^k, h^j) :  
	y^j h^k \in S_y\}\) is also equationally definable.

	We can now return to the equation in the ring of integers \(UV = W\).
	We encode this as the following system of equations with constraints:
\[ 
([X', Y'] = Z^c) \wedge 
(X, Y, Z \in \langle h \rangle) \wedge 
(X' \in S_x) \wedge
(Y' \in S_y) \wedge
((X', X)) \in R_x \wedge 
((Y', Y)) \in R_y.  
\]
Then 
$(h^i, h^j, h^k, x^ph^q, y^rh^s)$ 
is a solution to $(X,Y,Z,X',Y')$ in the above system of equations if and only if 
\[
[x^p h^q,y^rh^s] = h^{ck} \wedge 
(x^p h^q,h^i) \in R_x \wedge
(y^r h^s,h^j) \in R_y 
\] 
which holds if and only if $ h^{cpr} = h^{ck} $ and $ p=i $ and $r=j$. Since $c \neq 0$ this holds if and only if $ij=k$, $p=i$ and $r=j$.     
It follows that the set of solutions to \((X, Y, Z)\) 
in the above system of equations is 
\(\{(h^i, h^j, h^{ij}) :  
	i, j \in \Z\}\). 

We have proved that  ring of integers $\Z$ is interpretable in $G$ by finite systems of equations, and thus Hilbert's tenth problem over $\Z$  is reducible to the Diophantine problem in $G$. In particular, \(G\) has an undecidable Diophantine problem.
\end{proof}

\begin{lem}\label{lem:useful}
Let $G$ be a group and let $x,y,h \in G$ satisfy
\[
xh=hx, \quad yh=hy, \quad \mbox{and} \quad [x,y] = xyx^{-1}y^{-1} = h^d 
\]  
for some fixed integer $d \in \Z$. Then for all $i,j \in \Z$ we have 
\[
[x^i,y^j] = h^{dij}.
\]  
  \end{lem}
\begin{proof} 
From $xyx^{-1} = h^d y$ if follows that for all $j \in \Z$ we have 
\[
xy^j x^{-1} = (xy x^{-1})^j = (h^d y)^j = h^{dj} y^j 
\]   
with the last equality holding since $h$ commutes with $y$ by assumption. We shall now proceed by induction on \(|i|\) to show that
\begin{align}\label{NewOne} 
[x^i, y^j] & = h^{dij}.  
\end{align} 
If \(|i| = 0\), then both sides equal \(1\). Assume \(|i| > 0\), and inductively suppose the above statement holds, for all \(i' \in \Z\), with \(|i'| < |i|\). We begin with the case \(i > 0\). 
We have by induction  
\begin{align*} 
[x^i, y^j] 
& = x^{i - 1} (x y^j x^{-1}) x^{-(i-1)} y^{-j}  
 = x^{i - 1} h^{dj} y^j x^{-(i-1)} y^{-j} \\ 
& = x^{i - 1} y^j x^{-(i-1)} y^{-j} h^{dj}  
 = [x^{i-1}, y^j] h^{dj} = h^{(i-1)dj} h^{dj} = h^{ijd}. 
\end{align*} 
For \(i < 0\), note the in the above equalities, we only used the assumption \(i > 0\) to apply the inductive step. 
Thus the sequence of equalities above shows that for any integers $i$ and $j$ we have   
$[x^i, y^j] = [x^{i-1}, y^j] h^{dj}$. 
Replacing $i$ by $i+1$ in this equation and rearranging by moving  
$h^{dj}$ to the other side gives 
$[x^{i}, y^j] =[x^{i+1}, y^j] h^{-dj}  $ 
for all integers $i$ and $j$. Using induction, we thus obtain 
$[x^i, x^j] = h^{(i+1)jd} h^{-dj} = h^{ijd}$. 
This completes the proof that \eqref{NewOne} holds and thus completes the proof of the lemma. 
  \end{proof}

We now begin our analysis of the types of Seifert 3-manifold groups whose
base orbifold has 
zero Euler characteristic, as detailed in \cref{Seifert-Euler-qts}.

\begin{lem}[o1 - (1;)]\label{thm:3ManifoldGroups}
				\label{o1-1}
	  Let $b \in \mathbb{Z}$ and let 
	  \[
		  G = 
		  \langle 
		  u_1, v_1, h \mid 
		  [u_1,h]=1, \; [v_1,h]=1, \; [u_1,v_1]=h^b
		  \rangle. 
	  \]
	  Then the following are equivalent:
	  \begin{enumerate}
	  	\item \(G\) has a decidable Diophantine problem;
	  	\item \(b = 0\);
	  	\item \(G \cong \Z^3\).
	  \end{enumerate}
Furthermore, in the case that the above conditions are not satisfied then Hilbert's tenth problem over $\Z$  is reducible to the Diophantine problem in $G$.   
\end{lem}
\begin{proof}
We begin with the case that \(b = 0\). Then the presentation for \(G\) is clearly that for \(\Z^3\). Thus \(G\) has a decidable Diophantine problem. Therefore it remains to deal with the case that \(b \neq 0\).

First note that since \(h\) is central, every element of \(G\) can be expressed in the form \(uh^k\), for some $u \in FG(u_1,v_1)$ and some \(k \in \Z\). Moreover, 
we have 
$u_1 v_1 u_1^{-1} v_1^{-1} = [u_1,v_1] = h^b$ which implies  
$v_1 u_1 v_1^{-1} u_1^{-1} = h^{-b}$ and thus 
$v_1 u_1 = h^{-b} u_1 v_1 = u_1 v_1 h^{-b}$ as $h$ is central.      
This relation together with the fact that $h$ is central can then be applied to express every element of $G$ in the form $u_1^i v_1^j h^k$ for some $i,j,k \in \mathbb{Z}$. It is easy to see e.g. by considering the homomorphism to the quotient $G / \langle h \rangle \cong \Z \times \Z$ that the set of words     
$u_1^i v_1^j h^k$ with $i,j,k \in \mathbb{Z}$ is a set of normal forms for the elements of the group $G$.

We next show that \(C_G(u_1) = \langle u_1, h \rangle\) and \(C_G(v_1) = \langle v_1, h \rangle\), with a view to applying \cref{comm-undecidable}.  Since these are symmetric, we only show the former. Note that it is immediate from the presentation that \(\langle u_1, h \rangle \subseteq C_G(u_1)\).  
Conversely, suppose that $u_1^i v_1^j h^k \in C_G(u_1)$. Then 
$u_1^{i + 1} v_1^j h^k = u_1^i v_1^j h^k u_1 = u_1^i v_1^j u_1 h^k$. Now since $v_1 u_1 = u_1 v_1 h^{-b}$ it follows that 
$v_1^j u_1 = u_1 v_1^j h^{-jb}$ and hence 
\[
u_1^{i + 1} v_1^j h^k =
u_1^{i + 1} v_1^j h^{k-jb}.
\]
Thus $k = k-jb$,
and so $jb=0$ which, since $b \neq 0$, implies that $j=0$, and thus 
\(u_1^i v_1^j h^k  = u_1^i h^k  \in \langle u_1, h \rangle\), as required. We have thus shown that \(C_G(u_1) = \langle u_1, h \rangle\) and \(C_G(v_1) = \langle v_1, h \rangle\), and so these sets are equationally definable.  Since \(\langle h \rangle = C_G(u_1) \cap C_G(v_1)\) 
it follows that  
\(\langle h \rangle\) is also equationally definable.
Next, since $[u_1,v_1]=h^b$ and $h$ is central in $G$ it then follows from Lemma~\ref{lem:useful} that    $[u_1^i, v_1^j] = h^{bij}$ for all $i, j \in \Z$.

Now we can apply \cref{comm-undecidable} where in the statement of that result we take $x=u_1$, $y=v_1$, $S_x = S_{u_1} = C_G(u_1)$ and $S_y= S_{v_1} = C_G(v_1)$. Note that condition (2) of      
\cref{comm-undecidable} holds since 
\[
S_x / \langle h \rangle = 
\langle u_1, h \rangle / \langle h \rangle = 
\langle u_1 \rangle / \langle h \rangle
\] 
as $h$ is central, and similarly $S_y / \langle h \rangle = \langle y \rangle / \langle h \rangle$.  
Condition (1) of 
\cref{comm-undecidable} holds since $b \neq 0$ and the fact that condition (3) holds was observed above.    
Thus it follows from \cref{comm-undecidable} that the Diophantine problem in
\(G\) is undecidable when $b \neq 0$ and that Hilbert's tenth problem is
reducible to the Diophantine problem in \(G\).
	\end{proof}

	\begin{lem}[o2 - (1;)]
					\label{o2-1}
		Let \(b \in \Z\) and let \(G\) be the group with presentation	
		\[
				\langle u_1, v_1, h \mid u_1 h = h^{-1} u_1,
				v_1 h = h^{-1} v_1, [u_1, v_1] = h^b \rangle.
		\]
		Then \(G\) is virtually abelian, and hence has a decidable
		Diophantine problem.
	\end{lem}

\begin{proof}
We
have that \(G\) acts on \(\langle h \rangle\) by conjugation.  Since \(\langle
h \rangle \cong \Z\), by \cref{h-inf-order}, its automorphism group is
isomorphic to \(C_2\), and thus this induces a homomorphism \(\phi \colon G \to
\Aut(\langle h \rangle) \cong C_2\). In particular, \(\ker \phi\) has finite
index in \(G\).  We shall show that \(\ker \phi\) is abelian and hence $G$ is
virtually abelian. Note that by definition, \(\ker \phi = C_G(\langle h
\rangle)\), and so \(h\) is central in \(\ker \phi\).

We next show that \(\ker \phi\) is the set of elements expressible 
in the form $w h^k$ where $w \in \{u_1, u_1^{-1}, v_1,v_1^{-1}\}^*$ is a word of even length in the free monoid on these four generators, and $k \in \mathbb{Z}$.    
Indeed, since the four elements $\{u_1, u_1^{-1}, v_1,v_1^{-1}\}$ each conjugate \(h\) to \(h^{-1}\), any even length word over them will conjugate \(h\) to \(h\), and thus lie in \(C_G(h) = \ker \phi\).  Similarly, any word of the form $w h^k$ where $w \in \{u_1, u_1^{-1}, v_1,v_1^{-1}\}^*$ is a word of odd length 
will conjugate \(h\) to \(h^{-1}\), and thus not lie in the centraliser.  We have thus shown that 
\[
\ker \phi =
C_G(h) =  
\{ w h^k: 
\mbox{
$w \in \{u_1, u_1^{-1}, v_1,v_1^{-1}\}^*$ is a word of even length and $k \in \mathbb{Z}$
} \}.
\]
It follows that \(\ker \phi\) is generated by words over \(\{u_1, u_1^{-1}, v_1, v_1^{-1}\}\) of length \(2\) and \(h\). Removing inverses and freely reducing such words gives \(\ker \phi = \langle h, u_1^2, u_1 v_1, u_1 v_1^{-1}, u_1^{-1} v_1, u_1^{-1} v_1^{-1}, v_1^{-2} \rangle\). Noting that \(u_1 v_1^{-1} = u_1 v_1 \cdot v_1^{-2}\), \(u_1^{-1} v_1 = u_1^{-2} \cdot u_1 v_1\) and \(u_1^{-1} v_1^{-1} = u_1^{-2} \cdot u_1 v_1 \cdot v_1^{-2}\), gives that \(\ker \phi = \langle h, u_1^2, u_1 v_1, v_1^2 \rangle\).

		It now remains to show that \(\ker \phi = \langle h, u_1^2, u_1 v_1, v_1^2
		\rangle\) is indeed abelian. We have already shown that \(h\) is central,
		so it suffices to show that the other three generators all pairwise
		commute. Up to symmetry, it suffices to show \(u_1^2 \cdot u_1 v_1 
		= u_1 v_1 \cdot u_1^2\) and \(u_1^2 \cdot v_1^2 = v_1^2 \cdot u_1^2\).
Repeatedly applying the relation 
$u_1 v_1 = h^b v_1 u_1 $ we have  
\begin{align*}
			u_1^2 \cdot u_1 v_1
			&  = u_1^2 h^b (v_1 u_1) 
			 = u_1 (u_1 v_1) u_1 h^b 
			 = u_1 h^b (v_1 u_1^2) h^b 
			= u_1 v_1 \cdot u_1^2 
\end{align*}
and
\begin{align*}
						u_1^2 v_1^2
			& = u_1 h^b (v_1 u_1 v_1) 
			= u_1 v_1 u_1 v_1 h^{-b} 
			 = h^b (v_1 u_1^2 v_1) h^{-b} 
			 = v_1 u_1^2 v_1 \\
                        &  = v_1 u_1 h^b (v_1 u_1) 
			 = v_1 u_1 v_1 u_1 h^b 
		  	 = v_1 h^b (v_1 u_1^2) h^b 
			 = v_1^2 \cdot u_1^2.
		\end{align*}
		We have thus shown that \(\ker \phi\) is abelian, and so \(G\) is virtually abelian,
		and hence has a decidable Diophantine problem.
	\end{proof}

	\begin{lem}[o1 - (0; 2, 2, 2, 2)]
					\label{o1-0-2-2-2-2}
	Let \(b \in \Z\) and \(G\) be the group with presentation	
					\[
						G \cong \langle q_1, q_2, q_3, q_4, h \mid
						hq_i = q_ih, \; q_j^2 h = 1, \; q_1q_2q_3q_4 = h^b,  
\quad (1 \leq i,j \leq 4) \;
\rangle.
					\]
	Then the following are equivalent:
	\begin{enumerate}
					\item \(G\) has a decidable Diophantine problem;
					\item \(b = -1\);
					\item \(G\) is virtually abelian.
	\end{enumerate}
Furthermore, in the case that the above conditions are not satisfied then Hilbert's tenth problem over $\Z$  is reducible to the Diophantine problem in $G$.   
\end{lem}

\begin{proof}
					First note that \(h\) is central. Now define \(Q= G / \langle h
					\rangle\). We have that
					\[
							Q \cong \langle q_1, q_2, q_3, q_4 \mid q_j^2 = 1, q_1q_2q_3q_4 = 1
							\rangle.
					\]
From the second relation, we have that \(q_4\) is redundant, so \(Q \cong
\langle q_1, q_2, q_3 \mid q_j^2 = 1, (q_1q_2q_3)^2 = 1 \rangle\).  We first
show that \(Q\) is virtually abelian, with the subgroup \(\langle q_3q_1, q_2
q_3\rangle\) being abelian of index \(2\).  The fact that \(Q\) is virtually
\(\Z^2\) we know from \cref{Seifert-Euler-qts}, however we require the explicit
abelian subgroup.
Note that $q_1 q_2 q_3 =_Q (q_1 q_2 q_3)^{-1} =_Q q_3 q_2 q_1$ and so \(q_3 q_1 \cdot q_2 q_3 =_Q q_3 q_3 q_2 q_1 =_Q q_2 q_1 =_Q q_2 q_3 \cdot q_3 q_1\), and so \(\langle q_3 q_1, q_2 q_3 \rangle\) is indeed abelian. 
It is easy to see that the bijection $q_{i} \mapsto q_{i+1}$ where subscripts
are adjusted modulo $3$ defines an automorphism of the group $Q  
$, since the relations $(q_2 q_3 q_1)^2=1$ and $(q_3 q_1
q_2)^2=1$ both hold. We call this the cyclic shift automorphism.      Next,
note that $q_1 (q_2 q_3 q_1 q_2) q_3 =_Q 1$ implies $q_2 q_3 q_1 q_2^{-1} 
=_Q q_1 q_3$, and so 
\[
    q_3 (q_3 q_1) q_3^{-1} 
=_Q q_2 (q_3 q_1) q_2^{-1} 
=_Q q_1 (q_3 q_1) q_1^{-1} 
=_Q q_1 q_3 =_Q (q_3 q_1)^{-1} \in \langle q_3 q_1, q_2 q_3 \rangle.
\]
Applying the cyclic shift automorphism to these then gives 
\[
q_2 (q_2 q_3) q_2^{-1}, 
q_3 (q_2 q_3) q_3^{-1}, 
q_1 (q_2 q_3) q_1^{-1} \in 
\langle q_2 q_3, q_1 q_2 \rangle =   
\langle q_2 q_3, q_1 q_2, q_1 q_3 \rangle 
=
\langle q_3 q_1, q_2 q_3 \rangle. 
\]
Noting that \(q_1\), \(q_2\) and \(q_3\) are all finite order, we have thus
also shown that \(\langle q_3q_1, q_2 q_3 \rangle\) is closed under conjugation
by their inverses as well.  Thus \(\langle q_3 q_1, q_2 q_3 \rangle\) is normal
in \(Q\). So \(Q / \langle q_3 q_1, q_2 q_3 \rangle \cong \langle q_1, q_2, q_3
\mid q_j^2 =  1, (q_1q_2q_3)^2 =  1, q_3 q_1 =  q_2 q_3 =  1 \rangle\).
Applying Tietze transformations to remove the redundant generators \(q_1\) and
\(q_2\), gives the presentation \(\langle q_3 \mid q_3^2 = 1 \rangle\), and
thus \(\langle q_3 q_1, q_2 q_3\rangle\) has index \(2\) in $Q$, with \(\{1,
q_1\}\) being a transversal. 

Since $Q$ is virtually $\Z \times \Z$ by Theorem~\ref{Seifert-Euler-qts} and since 
$\langle q_3q_1, q_2 q_3\rangle$ is an abelian subgroup of $Q$ with index \(2\), it follows that  
$\langle q_3q_1, q_2 q_3\rangle \cong \Z \times \Z$.
It then follows that \(Q\) admits a normal form \(\{(q_3 q_1)^i (q_2q_3)^j q_1^\delta, \mid i, j \in \Z, \delta \in \{0, 1\}\}\).

We now return our attention to \(G\). 
Since $\langle h \rangle \cong \Z$ by \cref{h-inf-order}, we can lift the normal forms in $Q = G/ \langle h \rangle$ to the following set of normal forms for the group $G$:     
\[ 
\{h^k (q_3 q_1)^i (q_2q_3)^j q_1^\delta : i, j, k \in \Z, \delta \in \{0, 1\}\}.  
\] 
Additionally, removing the redundant generator $q_4 = q_1 q_2 q_3 h^{-b}$ from the presentation for \(G\) gives: 
\[ 
G \cong \langle q_1, q_2, q_3, h \mid h q_i = q_i h, 
\; q_j^2 h = 1, \; (q_1 q_2 q_3 h^{-b})^2 h = 1 \; (1 \leq i,j \leq 3) \rangle. 
\] 
In order to 
apply \cref{comm-undecidable}, we shall need to compute various commutators. To do this, we note that \(q_1 q_2 q_3 = h^{2b - 1} (q_1 q_2 q_3)^{-1}\). 
Observe that for every defining relation in the presentation, if we cyclically shift the subscripts on the $q_i$ letters (modulo \(3\)) we obtain relations that also hold in $G$. 
It follows that the map on the generators $\{q_1, q_2, q_3, h \}$ that fixes $h$ and maps $q_i \mapsto q_{i+1}$ with subscripts adjusted modulo $3$ defines an automorphism of $G$, that we call the cyclic shift automorphism.     
We next compute the action of each \(q_j\) on \(\langle q_3 q_1, q_2 q_3 \rangle\) by conjugation.  
Observe first that we have 
\begin{align} 
q_1 (q_2 q_3)^j q_1^{-1} & = (q_1 q_2 q_3 q_1^{-1})^j = (h^{2b - 1} q_3^{-1}
q_2^{-1} q_1^{-2})^j = (q_2 q_3 h^{-1})^{-j} h^{(2b - 1)j} = (q_2 q_3)^{-j}
h^{2bj}, \label{eq:2222conj1} \\ 
q_3 (q_2 q_3)^j q_3^{-1} & = 
q_2 (q_2 q_3)^j q_2^{-1} =
(q_3 q_2)^j = (q_3^{-1} h^{-1} q_2^{-1}h^{-1})^j = (q_2 q_3)^{-j} h^{-2j}. \label{eq:2222conj2}  
\end{align} 
Then applying the cyclic shift automorphism of $G$ we obtain 
\begin{align*}
q_2 (q_3 q_1)^j q_2^{-1} & = (q_3 q_1)^{-j} h^{2bj}, \quad 
q_3 (q_3 q_1)^j q_3^{-1} & =  (q_3 q_1)^{-j} h^{-2j}, \quad 
q_1 (q_3 q_1)^j q_1^{-1} & = (q_3 q_1)^{-j} h^{-2j}.  
\end{align*} 
Then 
\[
q_3q_1(q_2 q_3)q_1^{-1}q_3^{-1} 
= q_3 (q_2 q_3)^{-1} q_3^{-1} h^{2b} 
= (q_2 q_3) h^2 h^{2b}    
= (q_2 q_3) h^{(2(b+1))}    
\]
hence 
\[
[(q_3 q_1), (q_2 q_3)] = h^{2(b + 1)}.
\]
Then since $h$ is central in $G$ it then follows from Lemma~\ref{lem:useful} that 
\begin{align} 
\label{comm-eqn-2-2-2-2} [(q_3 q_1)^i, (q_2 q_3)^j] = h^{2(b + 1)ij}.  
\end{align} 
for all 
$i, j \in \Z$. 

We now separate into the cases of when \(b + 1 = 0\) and \(b + 1 \neq 0\), as per the statement of the result. 

\textbf{Case 1:} \(b + 1 = 0\). In this case \eqref{comm-eqn-2-2-2-2} tells us that \(q_3 q_1\) and \(q_2 q_3\) commute, and thus \(\langle q_3 q_1, q_2 q_3, h\rangle\) is abelian. 
The fact that \(h\) is central, together with the fact that conjugating \(q_3 q_1\) and \(q_2 q_3\) by \(q_1\), \(q_2\) and \(q_3\) (or their inverses) corresponds to multiplying the inverse of $q_3 q_1$ or $q_2 q_3$ by a power of \(h\) (as shown above) implies that \(\langle q_3 q_1, q_2 q_3, h \rangle\) is a normal subgroup of \(G\). 
Moreover, \(G / \langle q_3 q_1, q_2 q_3, h \rangle \cong Q / \langle q_3 q_1, q_2 q_3 \rangle \cong C_2\), and so \(G\) is virtually abelian, and has a decidable Diophantine problem.

\textbf{Case 2:} \(b + 1 \neq 0\). 
In this case we shall show that various subsets of $G$ are equationally definable and then use this to apply  
\cref{comm-undecidable} to show the Diophantine problem is undecidable. 

We show \(C_G(q_3 q_1) = \langle q_3 q_1, h \rangle\). It is immediate that \(C_G(q_3 q_2) \supseteq \langle q_3 q_1, h \rangle\). Let \(g  =  h^k (q_3 q_1)^i (q_2 q_3)^j q_1^\delta \in C_G(q_3 q_1)\). Then 
$q_3 q_1 \cdot g  = h^k (q_3 q_1)^{i + 1} (q_2 q_3)^j q_1^\delta$ and 
using Equations~\eqref{eq:2222conj1} and \eqref{eq:2222conj2}   
\begin{align*} 
g \cdot q_3 q_1 & = h^k (q_3 q_1)^i (q_2 q_3)^j q_1^\delta \cdot q_3 q_1 \\ 
 & = h^{k - 2 \delta} (q_3 q_1)^{i} (q_2 q_3)^j (q_3 q_1)^{(-1)^\delta} q_1^\delta 
 = h^{k -2\delta + 2(b+1)(-1)^{\delta+1}j} (q_3 q_1)^{i + (-1)^\delta} (q_2 q_3)^j q_1^\delta.  
\end{align*} 
The fact that these are equal first implies that \(\delta = 0\), by comparing exponents of \(q_3 q_1\), and then implies that \(j = 0\), by comparing exponents of \(h\).  Thus \(g  = h^k(q_3 q_1)^i \in \langle q_3 q_1, h \rangle\). This completes the proof that \(C_G(q_3 q_1) = \langle q_3 q_1, h \rangle\). 
Then applying the cyclic shift automorphism of $G$ we deduce that   
\(C_G(q_2 q_3) = \langle q_2 q_3, h \rangle\). 
As in the statement of \cref{comm-undecidable}, let $x=q_3 q_1$ and $y = q_2 q_3$ 
and set $S_x =\langle q_3 q_1, h \rangle $ and $S_y
=\langle q_2 q_3, h \rangle  $. As centralisers, $S_x$ and $S_y$ are
equationally
 definable subsets of $G$.
Condition (2) in \cref{comm-undecidable} holds since 
$\langle q_3 q_1, h \rangle / \langle h \rangle =  
\langle q_3 q_1 \rangle / \langle h \rangle$ 
and   
$\langle q_2 q_3, h \rangle / \langle h \rangle =  
\langle q_2 q_3 \rangle / \langle h \rangle$ 
as $h$ is central. Equation \eqref{comm-eqn-2-2-2-2} above shows that condition (1) in    
\cref{comm-undecidable} holds, while condition (3) holds since  
the intersection $\langle h \rangle = \langle q_3 q_1, h \rangle \cap  \langle q_2 q_3, h \rangle $ is equationally definable. 
Furthermore, \cref{h-inf-order} tells us that \(h\) has infinite order, thus all the hypotheses of 
\cref{comm-undecidable} are satisfied, and we conclude that $G$ has an undecidable Diophantine problem since Hilbert's tenth problem can be reduced to it.  
\end{proof}

\begin{lem}[o1 - (0; 3, 3, 3)]
				\label{o1-0-3-3-3}
	Let \(b \in \Z\) and \(\beta_1, \beta_2, \beta_3 \in \{1, 2\}\) and define 
	\[
		G = \langle q_1, q_2, q_3 \mid q_i h = h q_i,  q_j^3 h^{\beta_j} = 1, q_1 q_2 q_3 = h^b \rangle.
	\]
	Then the following are equivalent:
	\begin{enumerate}
		\item \(G\) has a decidable Diophantine problem;
		\item \(\beta_1 + \beta_2 + \beta_3 + 3b = 0\);
		\item \(G\) is virtually abelian.
	\end{enumerate}
Furthermore, in the case that the above conditions are not satisfied then Hilbert's tenth problem over $\Z$  is reducible to the Diophantine problem in $G$.   
\end{lem}

\begin{proof}
		We first note that \(\langle h \rangle\) is central, and hence normal. We thus
		consider \(Q = G / \langle h \rangle\) before dealing with \(G\). We
		show that \(Q\) is virtually abelian of rank \(2\) and construct a
		normal form for \(Q\). The former is already known 
		by \cref{Seifert-Euler-qts}, but in order to construct the latter,
		we want an explicit free abelian normal subgroup of a finite
		index, and a corresponding transversal. Note that
		\[
			Q \cong \langle q_1, q_2, q_3 \mid q_j^3 = 1, q_1 q_2 q_3 = 1 \rangle
\cong \langle q_1, q_2 \mid q_1^3 = q_2^3 = (q_1 q_2)^3 = 1 \rangle,
		\]
where we have removed the redundant generator \(q_3 = (q_1 q_2)^{-1} \).		
		Note that the third relation above, together with the fact \(q_j^{-1} =_Q
		q_j^2\) implies that \(q_1q_2q_1 q_2 =_Q q_2^2 q_1^2\). 
		Thus \(q_1 q_2^2 \cdot q_1^2 q_2 =_Q q_1 (q_1 q_2 q_1 q_2) q_2 =_Q q_1^2 q_2
		\cdot q_1 q_2^2\), and so \(\langle q_1^2 q_2, q_1 q_2^2 \rangle\) is
		abelian. 
Also $q_1q_2q_1 q_2 =_Q q_2^2 q_1^2$ implies 
$q_2q_1q_2 =_Q  
q_1^2 (q_1q_2q_1 q_2) =_Q q_1^2 q_2^2 q_1^2$.
Next, note that \(\langle q_1^2 q_2, q_1 q_2^2 \rangle\) is
normal in \(Q\), since
$q_1$ and $q_2$ have finite order in $Q$ and: 
		\begin{align}
q_2 \cdot q_1^2 q_2 \cdot q_2^{-1} & =_Q  
q_1 \cdot q_1^2 q_2 \cdot q_1^{-1} =_Q  
q_2 q_1^2 =_Q  (q_1 q_2^2)^{-1} \in \langle q_1^2 q_1, q_1q_2^2 \rangle \label{Qconj1} \\
q_1 \cdot q_1 q_2^2 \cdot q_1^{-1} & =_Q  q_1^2 q_2^2 q_1^2 =_Q  
q_2 q_1 q_2 =_Q  q_2 \cdot q_1 q_2^2 \cdot q_2^{-1} =_Q 
(q_1 q_2^2)^{-1} \cdot q_1^2 q_2 \in \langle q_1^2 q_2, q_1 q_2^2 \rangle. \label{Qconj2}
		\end{align}
	Moreover, \(Q / \langle q_1^2 q_2, q_1 q_2^2 \rangle\) has the presentation
	\[
		\langle q_1, q_2 \mid q_1^3 = q_2^3 = (q_1q_2)^3 = q_1^2 q_2 = q_2^2 q_1  =1 \rangle
		\cong \langle q_1 \mid q_1^3 = q_1^6 = q_1^9 = 1 \rangle \cong \langle q_1 \mid
		q_1^3 = 1 \rangle.
	\]
	Thus \(\{1, q_1, q_1^2\}\) is a transversal for \(\langle q_1^2 q_2, q_1
	q_2^2 \rangle\) in \(Q\). 
Since $Q$ is virtually $\Z \times \Z$ by Theorem~\ref{Seifert-Euler-qts} and $\langle q_1^2 q_2, q_1 q_2^2 \rangle$ is a finite index abelian subgroup of $Q$, it follows that $\langle q_3q_1, q_2 q_3\rangle \cong \Z \times \Z$.    
Hence \(Q\) admits the normal form \(\{(q_1^2 q_2)^i (q_1 q_2^2)^j q_1^\delta : i, j \in \Z, \delta \in \{0, 1, 2\}\}\). Since $\langle h \rangle \cong \Z$ by \cref{h-inf-order}, and $h$ is central in $G$, we can lift the normal forms in $Q = G/ \langle h \rangle$ to the following set of normal forms for the group $G$:     
	\[
		\{(q_1^2 q_2)^i (q_1 q_2^2)^j q_1^\delta h^k :  
		i, j, k \in \Z, \delta \in \{0, 1, 2\}\}
	\]
	We now consider commutators in \(G\). Let \(c = \beta_1 + \beta_2 + \beta_3 +
	3b\). 
Note that as \(\langle h \rangle\) is the kernel of the quotient map from \(G\) to \(Q\), and since $(q_1^2 q_2)$ and $(q_1 q_2^2)^j$ commute in $Q$,  we know that $[(q_1^2 q_2)^i, (q_1 q_2^2)^j] \in \langle h \rangle$ for all $i, j$. 

Next we compute the conjugate of \(q_1 q_2^2\) by
\(q_1^2 q_2\):
        \begin{align}
		    q_1^2 q_2 \cdot q_1 q_2^2 \cdot q_2^{-1} q_1^{-2}
		& = q_1 \cdot (q_1 q_2)^2 q_1^{-2}\nonumber 
		  = q_1 (h^b q_3^{-1})^2 q_1^{-2} \nonumber 
		  = q_1 (q_3^{-3} q_3) q_1^{-2} h^{2b} \nonumber \\
		& = q_1 (h^{\beta_3} q_2^{-1} q_1^{-1} h^b) q_1^{-2} h^{2b} \nonumber 
		  = q_1 q_2^{-1} q_1^{-3} h^{\beta_3 + 3b} \nonumber \\
		& = q_1 q_2^2 q_2^{-3} q_1^{-3} h^{\beta_3 + 3b}
		 = q_1 q_2^2 h^{\beta_1 + \beta_2 + \beta_3 + 3b}   
		\label{3-3-3-conj}
		 = q_1 q_2^2 h^c.
	\end{align}
Then since $[q_1^2 q_2,q_1 q_2^2]=h^c$ and $h$ is central in $G$ it then follows from Lemma~\ref{lem:useful} that 
\begin{align}
  		\label{comm-eqn-3-3-3}
			[(q_1^2 q_2)^i, (q_1 q_2^2)^j] = h^{cij}.
\end{align}
for all $i, j \in \Z$. 
	We now separately consider the cases when \(c = 0\) and when \(c \neq 0\),
	as per the statement of the result.

	\textbf{Case 1:} \(c = 0\). Then \eqref{comm-eqn-3-3-3} tells us that \(q_1^2 q_2\)
	and \(q_1 q_2^2\) and \(h\) all pairwise commute, and thus generate an
	abelian subgroup of \(G\). Moreover, the fact that \(\langle h \rangle\) is
	central, together with the fact that \(\langle q_1^2 q_2, q_1 q_2^2 \rangle\)
	is normal in \(Q\) imply that \(\langle  q_1^2 q_2, q_1 q_2^2, h \rangle\) is
	normal in \(G\). Thus \(G / \langle  q_1^2 q_2, q_1 q_2^2 , h\rangle \cong Q
	/ \langle q_1^2 q_2, q_1 q_2^2 \rangle \cong C_3\), and so this copy of
	\(\Z^3\) in \(G\) has index \(3\). Thus \(G\) is virtually abelian and has
	decidable Diophantine problem.

	\textbf{Case 2:} \(c \neq 0\).
	We will show that the following sets are
	equationally definable:
(1) the set \(\langle h \rangle\),
(2) the set \(\langle q_1^2 q_2, h \rangle\), and 
(3) the set \(\langle q_1 q_2^2, h \rangle\).
	For (1), it suffices to show that \(Z(G) = \langle h \rangle\). It is
	clear that \(\langle h \rangle \subseteq Z(G)\). To show the other containment,
	it suffices to show \(Z(Q) = \{1\}\). Now if \((q_1^2 q_2)^i (q_1 q_2^2)^j q_1^\delta
	\in Q\) is aribtrary, then
applying equations \eqref{Qconj1} and \eqref{Qconj2} we obtain:
	\begin{align*}
		[q_1, (q_1^2 q_2)^i (q_1 q_2^2)^j q_1^\delta]
		& =_Q q_1  (q_1^2 q_2)^i (q_1 q_2^2)^j q_1^\delta q_1^{-1}
		((q_1^2 q_2)^i (q_1 q_2^2)^j q_1^\delta)^{-1} \\
		& =_Q (q_1 q_2^2)^{-i} (q_1^2 q_2 \cdot (q_1 q_2^2)^{-1})^j
		q_1^\delta q_1^{-\delta} (q_1 q_2^2)^{-j} (q_1^2 q_2)^{-i} 
		 =_Q (q_1 q_2^2)^{-i -2j} (q_1^2 q_2)^{j -i}. 
	\end{align*}
	Then this commutator equals the identity if and only if
	\(-i -2j = 0\) and \(j - i = 0\), which holds if and only if \(i = j = 0\).
	Thus the only elements of \(Q\) that commute with \(q_1\) are \(1\), \(q_1\)
	and \(q_1^2\). In particular, the only central elements in \(Q\) can be
	\(1\) and \(q_1^2\). But \((q_1^2)^2 = q_1\), and so \(Z(Q) = \{1\}\),
	as required.

	For (2) (note that (3) is symmetric), it suffices to show that
	\(\langle q_1^2 q_2, h \rangle = C_G(q_1^2 q_2)\). Again, it is clear
	that \(\langle q_1^2 q_2, h \rangle \subseteq C_G(q_1^2 q_2)\).  For the other containment, we again consider the projection into
	\(Q\). We shall prove that \(C_Q(q_1^2 q_2)\)
is contained in
	\(\langle q_1^2 q_2, q_1 q_2^2 \rangle \subseteq Q\). Let
	\((q_1^2 q_2)^i (q_1 q_2^2)^j q_1^\delta \in C_Q(q_1^2 q_2)\). 
Then applying Equations \eqref{Qconj1} and \eqref{Qconj2} 
\begin{align*}
					1 & =_Q [q_1^2 q_2, (q_1^2 q_2)^i (q_1 q_2^2)^j q_1^\delta] 
						 =_Q (q_1^2 q_2)^{i + 1}(q_1 q_2^2)^j q_1^\delta
						(q_1^2 q_2)^{-1} q_1^{-\delta} (q_1 q_2^2)^{-j} (q_1^2 q_2)^{-i}\\
						& =_Q q_1^2 q_2 \cdot q_1^\delta (q_1^2 q_2)^{-1} q_1^{-\delta} 
						 =_Q [q_1^2 q_2, q_1^\delta].
	\end{align*}
	From the proof of (1), we know this is trivial if and only if \(\delta = 0\).
	It follows that \(C_Q(q_1^2 q_2) \subseteq \langle q_1^2 q_2, q_1 q_2^2
	\rangle\), and so \(C_G(q_1^2 q_2) \subseteq \langle q_1^2 q_2, q_1 q_2^2, h
	\rangle\). Now let \((q_1^2 q_2)^i (q_1 q_2^2)^j h^k \in C_G(q_1^2 q_2)\).
	Then, using \eqref{3-3-3-conj}
	\begin{align*}
			1 & = [q_1^2 q_2, (q_1^2 q_2)^i (q_1 q_2^2)^j h^k] = (q_1^2 q_2)^{i + 1} (q_1 q_2^2)^j h^k (q_1^2 q_2)^{-1}
				h^{-k} (q_1 q_2^2)^{-j} (q_1 q_2^2)^{-i} \\
				& = (q_1^2 q_2)^{i} (q_1^2 q_2 \cdot q_1 q_2^2 \cdot q_2^{-1} q_1^{-2})^j
				(q_1 q_2^2)^{-j} (q_1 q_2^2)^{-i} 
				= (q_1^2 q_2)^{i +} (q_1 q_2^2 h^c)^j
				(q_1 q_2^2)^{-j} (q_1 q_2^2)^{-i} 
				= h^{cj}.
	\end{align*}
	Using the fact that \(c \neq 0\), we obtain that \(j = 0\), and so
	\(C_G(q_1^2 q_2) = \langle q_1^2 q_2, h \rangle\), as required.
Now setting $x=q_1^2 q_2$ and $y = q_1 q_2^2$ in the statement of \cref{comm-undecidable}, and setting $S_x =\langle q_1^2 q_2, h \rangle $ and $S_y =\langle q_1 q_2^2, h \rangle  $, it follows that $S_x$ and $S_y$ are equationally
 definable subsets of $G$.
Condition (2) in \cref{comm-undecidable} holds since 
$\langle q_1^2 q_2, h \rangle / \langle h \rangle =  
\langle q_1^2 q_2 \rangle / \langle h \rangle$ 
and   
$\langle q_1 q_2^2, h \rangle / \langle h \rangle =  
\langle q_1 q_2^2 \rangle / \langle h \rangle$ 
as $h$ is central. Equation \eqref{comm-eqn-2-2-2-2} above shows that condition (1) in    
\cref{comm-undecidable} holds, while condition (3) holds since we proved above that $Z(G) = \langle h \rangle$.  
Moreover, \cref{h-inf-order} tells us that \(h\) has infinite order, and so the hypotheses of 
\cref{comm-undecidable} are satisfied, and we conclude that the Diophantine problem in \(G\) is undecidable, as one can reduce Hilbert's tenth problem over
the ring of integers to it.
\end{proof}

\begin{lem}[o1 - (0; 2, 4, 4)]
				\label{o1-0-2-4-4}
	Let
	\[
		G = \langle q_1, q_2, q_3, h \mid q_i h = q_i h, \; q_1^2 h = 1, \;
		q_2^4 h^{\beta_2} = 1, \; q_3^4 h^{\beta_3} = 1, \; q_1 q_2 q_3 = h^b \rangle,
	\]
	where \(\beta_2, \beta_3 \in \{1, 2, 3\}\) and \(b \in \Z\). Then the following are equivalent:
	\begin{enumerate}
		\item \(G\) has a decidable Diophantine problem;
		\item \(2 + \beta_2 + \beta_3 + 4b = 0\);
		\item \(G\) is virtually abelian.
	\end{enumerate}
Furthermore, in the case that the above conditions are not satisfied then Hilbert's tenth problem over $\Z$  is reducible to the Diophantine problem in $G$.   
\end{lem}

\begin{proof} Note that \(\langle h \rangle\) is central, and hence normal. Before we consider \(G\) directly, we consider \(Q = G / \langle h \rangle\), since we can use this to construct a normal form for \(G\). Note that \(Q\) has the presentation \(\langle q_1, q_2, q_3 \mid q_1^2 = 1, q_2^4 = 1, q_3^4 = 1, q_1 q_2 q_3 = 1\rangle\). Removing the redundant generator \(q_1\) gives that \(Q \cong \langle q_2, q_3 \mid q_2^4 = 1, q_3^4 = 1, (q_2 q_3)^2 = 1\rangle\). We know from Theorem~\ref{Seifert-Euler-qts} that \(Q\) is virtually \(\Z^2\). We claim that \(\langle q_2^3 q_3, q_2 q_3^3 \rangle\) generates a finite index abelian subgroup of $Q$. It will then follow that $\langle q_2^3 q_3, q_2 q_3^3 \rangle \cong \Z^2$ since $\Z^2$ the only \(2\)-generated abelian group that is virtually \(\Z^2\). We begin by showing this subgroup is abelian. We have
		\[
			q_2^3 q_3 \cdot q_2 q_3^3 =_Q q_2^2 (q_2 q_3)^2 q_3^2
			=_Q q_2^2 q_3^2, \qquad
			q_2 q_3^3 \cdot q_2^3 q_3 =_Q q_2 (q_2 q_3)^{-1} q_3
			=_Q q_2 (q_2 q_3) q_3 =_Q q_2^2 q_3^2.
		\]
Thus \(\langle q_2^3 q_3, q_2 q_3^3 \rangle\) is abelian. We next show that this subgroup is normal. Since $q_2$ and $q_3$ have finite order in $Q$ it suffices show closure under conjugation be each of these elements. We have:  
\begin{align*}
						q_2 (q_2^3 q_3) q_2^{-1} & =_Q q_3 q_2^{-1} =_Q (q_2 q_3^3)^{-1}  & 
						q_2 (q_2 q_3^3) q_2^{-1} & =_Q q_2^2 q_3^2 q_3 q_2^{3}
						=_Q q_2^3 q_3 \cdot q_2 q_3^3 \cdot (q_2 q_3^3)^{-1} =_Q q_2^3 q_3
                \end{align*}
where the last equality is computed using the first of the equalities in the previous displayed equation.  
The remaining statements needed to show that \(\langle q_2^3 q_3, q_2 q_3^3 \rangle \trianglelefteq Q\) follow by symmetry since the mapping that swaps $q_2$ and $q_3$ clearly induces an automorphism of the group $Q$. We can use the presentation of \(Q\) to deduce that \(Q / \langle q_2^3 q_3, q_2 q_3^3 \rangle \cong
		\langle q_2, q_3 \mid q_2^4 = 1, q_3^4 = 1, (q_2 q_3)^2 = 1, q_2^3 q_3 = 1,
		q_2 q_3^3 = 1 \rangle\). The final relation is equivalent (using \(q_3^4=
		1\)) to \(q_2 = q_3\). We can thus remove the redundant generator \(q_2\) to conclude that this quotient is isomorphic to \(\langle q_2 \mid q_2^4 =
		1\rangle\), which is finite. It follows that \(\langle q_2^3 q_3, q_2 q_3^3
		\rangle\) has finite index in $Q$, as required, with \(\{1, q_2, q_2^2, q_2^3\}\) being a transversal. We conclude that \(Q\) has the normal form \(\{(q_2^3 q_3)^i (q_2 q_3^3)^j q_2^\delta : i, j \in \Z, \delta \in \{0, 1, 2, 3\}\}\). Using the fact that \(G\) is an extension of \(\langle h \rangle\) by \(Q\), that $h$ is central in $G$, and the fact that \(\langle h \rangle \cong \Z\), by \cref{h-inf-order}, we have that \(G\) admits the normal form
    \[
			\{h^k (q_2^3 q_3)^i (q_2 q_3^3)^j q_2^\delta : i, j, k \in \Z, \delta \in \{0, 1, 2, 3\}\}.
    \]
Observe that in $G$ we have
$(q_2 q_3)^2 = (q_1^{-1} h^b)^2 = q_1^{-2} h^{2b} = h^{1+2b}$.
With a view to applying Lemma~\ref{lem:useful} we compute \(q_2^3 q_3 \cdot q_2 q_3^3 \cdot (q_2^3 q_3)^{-1}\) in $G$ obtaining:              
		\begin{align*}
						q_2^3 q_3 \cdot q_2 q_3^3 \cdot (q_2^3 q_3)^{-1}
						& = q_2^3 q_3 q_2 q_3^2 q_2 q_2^{-4} 
						 = q_2^2 (q_2 q_3)^2 q_3 q_2 h^{\beta_2}  =q_2^2 h^{1+2b} q_3 q_2 h^{\beta_2}  
						 = q_2^2 q_3 q_2 h^{1 + \beta_2 + 2b} \\
						& = q_2 (q_2 q_3)^2 q_3^{-1} h^{1 + \beta_2 + 2b} =
                                                  q_2 h^{1+2b}  q_3^{-1} h^{1 + \beta_2 + 2b}  
						 = q_2 q_3^3 q_3^{-4} h^{2 + \beta_2 + 4b} 
						 = q_2 q_3^3 h^{2 + \beta_2 + \beta_3 + 4b}.
		\end{align*}

It then follows from Lemma~\ref{lem:useful} that 
\begin{align}
[(q_2^3 q_3)^i, (q_2 q_3^3)^j] = h^{cij} \label{Eq_2332}
\end{align}
for all $i,j \in \Z$ where $c = 2 + \beta_2 + \beta_3 + 4b$.

We now separate into two cases: \(c = 0\) and \(c \neq 0\).

\textbf{Case 1:} \(c = 0\). In this case, \([q_2^3 q_3, q_2 q_3^3] = 1\), and so \(A  =  \langle q_2^3 q_3, q_2 q_3^3, h \rangle\) is an abelian subgroup of \(G\) satisfying
\([G : A] = [G / \langle h \rangle : A / \langle h \rangle] = [Q: \langle q_2^3 q_3, q_2 q_3^3 \rangle] = 4\). Thus \(G\) is virtually abelian, as required.

\textbf{Case 2:} \(c \neq 0\). It now suffices to show that \(\langle h \rangle\), \(\langle q_2^3 q_3, h \rangle\) and \(\langle q_2 q_3^3, h \rangle\) are all equationally definable in order to apply \cref{comm-undecidable}. Since \(\langle h \rangle = \langle q_2^3 q_3, h \rangle \cap \langle q_2 q_3^3, h \rangle\), it suffices to show the latter two are equationally definable. But these are symmetric, so it will suffice to show that \(C_G(q_2^3 q_3) = \langle q_2^3 q_3, h \rangle\).

Clearly, \(C_G(q_3^3 q_3) \supseteq \langle q_2^3
		q_3, h \rangle\). For the converse inclusion let \(g  =  h^k (q_2^3 q_3)^i (q_2 q_3^3)^j
		q_2^\delta \in C_G(q_2^3 q_3)\). Then \(g / \langle h \rangle \in
		C_Q(q_1^3 q_2)\). So consider \((q_2^3 q_3)^i (q_2 q_3^3)^j q_2^\delta
		= g / \langle h \rangle \in Q\). Then, using the conjugates of \(\langle q_2^3 q_3, q_2 q_3^3 \rangle\) in \(Q\) we computed earlier, we have

\vspace{-10mm}

		\begin{align*}
			(q_2^3 q_3)^{i + 1} (q_2 q_3^3)^j q_2^\delta
			& =_Q q_2^3 q_3 \cdot (q_2^3 q_3)^i (q_2 q_3^3)^j q_2^\delta 
			 =_Q (q_2^3 q_3)^i (q_2 q_3^3)^j q_2^\delta
\cdot q_2^3 q_3 
			 =_Q \begin{cases}
											(q_2^3 q_3)^{i + 1} (q_2 q_3^3)^j q_2^\delta & \delta = 0 \\
											(q_2^3 q_3)^{i} (q_2 q_3^3)^{j - 1} q_2^{\delta} & \delta = 1 \\
											(q_2^3 q_3)^{i - 1} (q_2 q_3^3)^j q_2^{\delta} & \delta = 2 \\
											(q_2^3 q_3)^{i} (q_2 q_3^3)^{j + 1} q_2^{\delta} & \delta = 3.
						\end{cases}
		\end{align*}
The only case where we have equality is \(\delta = 0\). So \(g = h^k (q_2^3
		q_3)^i (q_2 q_3^3)^j\). Then applying Equation \ref{Eq_2332} we obtain 
		\begin{align*}
						h^k (q_2^3 q_3)^{i + 1} (q_2 q_3^3)^j
						& = q_2^3 q_3 \cdot g 
						 = g \cdot q_2^3 q_3 
						 = h^k (q_2^3 q_3)^i (q_2 q_3^3)^j (q_2^3 q_3) 
						 = h^{k + cj} (q_2^3 q_3)^{i+1} (q_2 q_3^3)^j. 
		\end{align*}
                Thus \(j = 0\), and \(g \in \langle q_2^3 q_3, h \rangle\), as required.
Now setting $x=q_2^3 q_3$ and $y = q_2 q_3^3$ in the statement of
\cref{comm-undecidable} and setting $S_x =\langle q_2^3 q_3, h \rangle $ and
$S_y =\langle q_2 q_3^3, h \rangle  $ it follows that $S_x$ and $S_y$ and
$\langle h \rangle$ are equationally definable subsets of $G$. Together with
Equation \eqref{Eq_2332} and the fact that $h$ is central and has infinite
order by Lemma~\ref{Seifert-pres}, it then follows that all the conditions of
\cref{comm-undecidable} are satisfied. 
Thus Hilbert's tenth problem is reducible to the Diophantine problem in
\(G\), and so the latter is undecidable.
\end{proof}

\begin{lem}[o1 - (0; 2, 3, 6)]
				\label{o1-0-2-3-6}
	Let
	\[
		G = \langle q_1, q_2, q_3, h \mid q_i h = h q_i, q_1^2 h = 1,
		q_2^3 h^{\beta_2} = 1, q_3^6 h^{\beta_3} = 1, q_1 q_2 q_3 = h^b \rangle,
	\]
	where \(\beta_2 \in \{1, 2\}\), \(\beta_3 \in \{1, 2, 3, 4, 5\}\) and \(b \in \Z\). Then the
	following are equivalent:
	\begin{enumerate}
		\item \(G\) has a decidable Diophantine problem;
		\item \(2 \beta_2 + \beta_3 + 6b + 3 = 0\);
		\item \(G\) is virtually abelian.
	\end{enumerate}
Furthermore, in the case that the above conditions are not satisfied then Hilbert's tenth problem over $\Z$  is reducible to the Diophantine problem in $G$.   
\end{lem}

\begin{proof}
		Note that \(\langle h \rangle\) is central, and hence normal. As with earlier lemams, before we
		consider \(G\) directly, we consider \(Q = G / \langle h \rangle\),
		since we can use this to construct a normal form for \(G\). Note that
		\(Q\) has the presentation \(\langle q_1, q_2, q_3 \mid q_1^2 = 1,
		q_2^3 = 1, q_3^6 = 1, q_1 q_2 q_3 = 1\rangle\). Removing the redundant
		generator \(q_3\) gives that \(Q \cong \langle q_1, q_2 \mid q_1^2 = 1,
		q_2^3 = 1, (q_1 q_2)^6 = 1\rangle\). 
We claim that
		\(\langle q_1q_2q_1q_2^2, q_2^2 q_1 q_2  q_1 \rangle\) is a
		finite index abelian subgroup of $Q$ which would then imply that 
\(\langle q_1q_2q_1q_2^2, q_2^2 q_1 q_2  q_1 \rangle \cong \Z^2\) since by  
\cref{Seifert-Euler-qts} we know that  \(Q\) is virtually \(\Z^2\).
We begin by showing this subgroup is abelian. We have
		\begin{align*}
			q_1q_2q_1q_2^2 \cdot q_2^2 q_1 q_2  q_1
			& =_Q  (q_1 q_2)^3 q_1 
			 =_Q  (q_2^2 q_1)^3 q_1 
			 =_Q  q_2^2 q_1 q_2^2 q_1 q_2^2 
			 =_Q  q_2^2 q_1 q_2 q_1 \cdot q_1q_2q_1q_2^2 
		\end{align*}

Thus \(x  =  q_1 q_2 q_1 q_2^2\) and \(y  =  q_2^2 q_1 q_2
		q_1\) commute and \(\langle x, y \rangle \leq Q\) is abelian. We next show
		that \(\langle x, y \rangle\) is normal. We have 
		\begin{align*}
			q_1 x q_1^{-1} & =_Q q_2 q_1 q_2^2 q_1 =_Q  (q_1 q_2 q_1 q_2^2)^{-1}
			=_Q  x^{-1}, 
			q_2 x q_2^{-1}  =_Q  q_2 q_1 q_2 q_1 q_2
			=_Q  (q_2^2 q_1 q_2^2 q_1 q_2^2)^{-1} =_Q  (yx)^{-1} \\
			q_1 y q_1^{-1} & =_Q  q_1 q_2^2 q_1 q_2 =_Q  (q_2^2 q_1 q_2 q_1)^{-1} =_Q  y^{-1}, 
			q_2 y q_2^{-1}  =_Q  q_1 q_2 q_1 q_2^2 =_Q  x.
		\end{align*}
	which proves that \(\langle x, y \rangle\) is normal since $q_1$ and $q_2$ have finite order in $Q$.    	
We now note that the quotient of \(Q\) by \(\langle x, y \rangle\)
		is
		\begin{align*}
			& \langle q_2, q_2 \mid q_1^2 = 1, q_2^3 = 1, (q_1 q_2)^6 =
			1, q_1 q_2 q_1 q_2^2 = 1, q_2^2 q_1 q_2 q_1 = 1\rangle \\
			& \cong \langle q_1, q_2 \mid q_1^2 = 1, q_2^3 = 1, (q_1 q_2)^6 = 1, q_1  q_2 q_1 q_2^2 = 1, q_2^2 q_1 q_2 q_1 = 1, q_2^{-1} = (q_1 q_2)^2, q_1 = (q_1q_2)^3\rangle \\
			& \cong \langle q_1, q_2, z \mid q_1^2 = 1, q_2^3 = 1, (q_1 q_2)^6 = 1, q_1  q_2 q_1 q_2^2 = 1, q_2^2 q_1 q_2 q_1 = 1, q_2^{-1} = (q_1 q_2)^2, q_1 = (q_1q_2)^3, z = q_1 q_2 \rangle \\
			& \cong \langle z \mid z^6 = 1 \rangle.
		\end{align*}
In the first step above from $(q_1q_2q_1q_2)q_2=1$ we derive $q_2^{-1} = (q_1q_2)^2$, and then from this it follows that $1 = q_2(q_1 q_2)^2$ implying $q_1 = (q_1q_2)^2$. That implies that the group is cyclic generated by $z = q_1q_2$.

Thus \(\{(q_1 q_2)^\delta : \delta \in \{0, \ldots, 5\}\}\) is a transversal for \(\langle x, y \rangle \in Q\). With \(z = q_1 q_2\), it follows that \(Q\) has the normal form \(x^i y^j z^\delta\), where \(i, j \in \Z\) and \(\delta \in \{0, \ldots, 5\}\). Lifting everything back to \(G\), noting that \(\langle h \rangle \cong \Z\) by \cref{h-inf-order}, it follows that  
\[
\{ x^i y^j z^{\delta} h^k: 
\mbox{\(i, j, k \in \Z\) and \(\delta \in \{0, \ldots, 5\}\}\)} 
\]
is a set of normal forms for $G$, where $x =q_1 q_2 q_1 q_2^2$, $y=q_2^2 q_1 q_2 q_1$ and $z=q_1 q_2$.      
We next compute \(xyx^{-1}\) in \(G\), which we will need to compute \([x^i, y^j]\), for \(i, j \in \Z\). 
Note that in $G$ we have 
\[
(q_1 q_2)^6 = (h^b q_3^{-1})^6 = h^{6b}q_3^{-6} = h^{6b + \beta_3}.
\] 
Now in the group $G$ we have
		\begin{align*}
						xyx^{-1} & = q_1 q_2 q_1 q_2^2 \cdot q_2^2 q_1 q_2 q_1 \cdot q_2^{-2} q_1^{-1}
						q_2^{-1} q_1^{-1}
										  & & = q_1 q_2 q_1 q_2 q_2^3 q_1 q_2 q_1 q_2 q_2^{-3} q_1 q_1^{-2} q_2^2
										 q_2^{-3} q_1 q_1^{-2} \\
										 & = q_1 q_2 q_1 q_2 h^{-\beta_2} q_1 q_2 q_1 q_2 h^{\beta_2} q_1 h q_2^2 h^{\beta_2} q_1 h
							  & & = (q_1 q_2)^5 q_2 q_1 h^{\beta_2 + 2} \\
							 & = (q_1 q_2)^6 q_2^{-1} q_1^{-1} q_2 q_1 h^{\beta_2 + 2} 
							  & & = q_2^{-1} q_1^{-1} q_2 q_1 h^{\beta_2 + 2 + 6b + \beta_3} \\
							 & = (q_2^{2}h^{\beta_2}) (q_1 h) q_2 q_1 h^{\beta_2 + 2 + 6b + \beta_3} 
							  & & = y h^{2\beta_2 + \beta_3 + 6b + 3}.
		\end{align*}
This shows that $[x,y] = h^c$ where $c = 2 \beta_2 + \beta_3 + 6b + 3$. It then follows from   
Lemma~\ref{lem:useful} that in $G$:
		\begin{align}\label{eqn:xycij}
			[x^i, y^j] & = h^{cij}
		\end{align}
for all $i,j \in \Z$.  
We now consider the cases when \(c = 0\) and \(c \neq 0\) separately.

\textbf{Case 1:} \(c = 0\). 
In this case, \([x, y] = h^c = 1\), and so 
\(A  =  \langle x, y, h \rangle\) 
is an abelian subgroup of \(G\) satisfying
\([G:A] = [G / \langle h \rangle : A / \langle h \rangle] = [Q:\langle x, y \rangle] = 6\), and so \(A\) is finite index in \(G\), and \(G\) is virtually abelian and thus has a decidable Diophantine problem, as required.

\textbf{Case 2:} \(c \neq 0\). We aim to invoke \cref{comm-undecidable}, which requires us to prove that \(\langle h \rangle\), \(\langle h, x \rangle\) and \(\langle h, y \rangle\) are all equationally definable. We have \(\langle h \rangle = \langle h, x \rangle \cap \langle h, y \rangle\), so to show \(\langle h \rangle\) is equationally definable, it suffices to show that the latter two sets are. We begin by showing \(C_G(x) = \langle x, h \rangle\). Clearly, \(C_G(x) \supseteq \langle x, h \rangle\), so let \(g  =  x^i y^j z^\delta h^k \in C_G(x)\). Then \(g / \langle h \rangle \in C_Q(x)\), and so using the 
calculations above of conjuation in \(Q\) of \(x\) and \(y\) by \(q_1\) and \(q_2\) we have 
$q_1q_2 x =_Q xy q_1 q_2$ and $q_1 q_2 y =_Q x^{-1} q_2 q_2$, which we can then use to show 
		\begin{align*}
			x^{i + 1} y^j z^\delta 
			 =_Q x^i y^j z^\delta x 
			 =_Q x^i y^j (q_1 q_2)^\delta x 
			 =_Q \begin{cases}
							x^{i + 1} y^j z^\delta & \delta = 0 \\
							x^{i + 1} y^{j + 1} z^\delta & \delta = 1 \\
							x^i y^{j + 1} z^\delta & \delta = 2 \\
							x^{i - 1} y^j z^\delta & \delta = 3 \\
							x^{i - 1} y^{j - 1} z^\delta & \delta = 4\\
							x^i y^{j - 1} z^\delta & \delta = 5.
						\end{cases}
		\end{align*}
		The only case without a contradiction is \(\delta = 0\), and so \(g = x^i y^j h^k\).
		Thus working in \(G\):
		\begin{align*}
			x^{i + 1} y^j h^k & = xg 												
             = gx 
												 = x^i y^j h^k \cdot x 
												 = x^{i + 1} y^j h^{k - cj}.
		\end{align*}
		Thus \(j = 0\) and \(g = x^i h^k \in \langle x, h \rangle\). We can
		conclude that \(C_G(x) = \langle x, h \rangle\). 
Next note that by definition we have $y = q_2 x q_2^{-1}$ in $G$. Thus since $h$ is central in $G$ we have  
\[
C_G(y) = 
C_G(q_2 x q_2^{-1}) =  
q_2^{-1} C_G(x) q_2 =
\langle q_2^{-1} x q_2, q_2^{-1} h q_2 \rangle = 
\langle y, h \rangle.
\] 
Now setting 
$x =q_1 q_2 q_1 q_2^2$ and $y=q_2^2 q_1 q_2 q_1$ 
in the statement of \cref{comm-undecidable} and letting $S_x =\langle x, h \rangle $ and $S_y =\langle y, h \rangle  $ it follows that $S_x$ and $S_y$ and $\langle h \rangle$ are Equationally definable subsets of $G$. Together with equation \eqref{eqn:xycij} and the fact that $h$ is central and has infinite order by Lemma~\ref{Seifert-pres}, it then follows that the conditions of \cref{comm-undecidable} are satisfied. 
Thus Hilbert's tenth problem reduces to the Diophantine problem in \(G\) which is therefore undecidable, completing the proof of the lemma. 
\end{proof}

\begin{lem}[n1-3, (2;)]
				\label{n123-2}
    Let \(b \in \Z\) and \(\epsilon_1, \epsilon_2 \in \{1, -1\}\),
    and let \(G\) be the group with the presentation
    \[   
        \langle v_1, v_2, h \mid v_i h = h^{\epsilon_i} v_i, \; v_1^2 v_2^2 = h^b
        \rangle.
    \]
    Then
    \begin{enumerate}
        \item If \(\epsilon_1 = \epsilon_2 = 1\) or \(\{\epsilon_1, \epsilon_2\}
                = \{-1, 1\}\), then \(G\) is virtually abelian;
        \item If \(\epsilon_1 = \epsilon_2 = -1\), then the following
                are equivalent:
                \begin{enumerate}
                    \item \(b = 0\);
                    \item \(G\) is virtually abelian;
                    \item \(G\) has a decidable Diophantine problem.
                \end{enumerate}
Furthermore, in the case that the above conditions are not satisfied then Hilbert's tenth problem over $\Z$  is reducible to the Diophantine problem in $G$.   
    \end{enumerate}  
\end{lem}

\begin{proof}
	We first consider \(Q = G / \langle h \rangle\). We have that \(Q = \langle
	v_1, v_2 \mid v_1^2 v_2^2 = 1 \rangle\). This is the Klein bottle group,
	which is virtually \(\Z^2\), with 
\(\Z^2 \cong \langle v_1^2, v_1 v_2\rangle =\langle v_1^2, v_1 v_2, v_2v_1, v_2^{2} \rangle \leq Q  \) 
being a
	subgroup of index \(2\), with transversal \(\{1, v_1\}\). Thus \(Q\) admits
	the normal form \(\{v_1^{2i} (v_1 v_2)^j v_1^\delta : i, j \in \Z, \delta
					\in \{0, 1\}\}\), and so, since $h$ has infinite order by \cref{h-inf-order},  \(G\) admits a set of normal forms  \(\{h^k v_1^{2i}
					(v_1 v_2)^j v_1^\delta : i, j, k \in \Z, \delta \in \{0, 1\}\}\).
					We now separate into the three cases from the statement of the lemma.

	\textbf{Case 1:} \(\epsilon_1 = \epsilon_2 = 1\). In this case, we have $v_2^2v_1^2=h^b$ in $G$ and then in $G$ we have: 
	\begin{align*}
					v_1^2 \cdot v_1 v_2 
					& = v_1 v_1^2 v_2^2 v_2^{-1} 
					 = v_1 h^b v_2^{-1}
					 = v_1 h^b (v_2 v_1^2  h^{-b}) 
					 = v_1 v_2 \cdot v_1^2.
	\end{align*}
	Since \(h\) is central, we have thus shown that \(\langle v_1^2, v_1v_2,
	h \rangle \leq G\) is abelian. 
The index in $G$ of this subgroup is   
\([G : \langle v_1^2, v_1 v_2, h \rangle]
	= [Q: \langle v_1^2, v_1 v_2 \rangle] = 2\), and so \(G\) is virtually
	abelian in this case.

\textbf{Case 2:} \(\{\epsilon_1, \epsilon_2\} = \{-1, 1\}\). We shall consider the case \(\epsilon_1 = 1\) and \(\epsilon_2 = -1\). The proof for the case \(\epsilon_2 = 1\) and \(\epsilon_1 = -1\) is the same but with the roles of $v_1$ and $v_2$ interchanged.     
Unlike Case 1, we will show a different (smaller) subgroup of $G$ is abelian and has finite index, namely the subgroup \(\langle v_1^2, (v_1 v_2)^2, h \rangle\). Firstly, note that \(h\) is indeed central in \(\langle v_1^2, (v_1 v_2)^2, h \rangle\), and so it suffices to check that in $G$ the remaining two generators commute:
	\begin{align*}
					v_1^2 \cdot (v_1 v_2)^2 v_1^{-2} 
					& = v_1 v_1^2 v_2^2 v_2^{-1} v_1^{-1} v_1^2 v_2^2 v_2 v_2^{-2} v_1^{-2} \\
					& = v_1 h^b v_2^{-1} v_1^{-1} h^b v_2 h^{-b } 
					  = h^{-b} v_1 v_2^{-1} v_1^{-1} v_2 \\
					& = h^{-b} v_1 v_2 v_2^{-2} v_1^{-2} v_1 v_2 
					 = h^{-b} v_1 v_2 h^{-b} v_1 v_2 
					= (v_1 v_2)^2.
	\end{align*}
Hence \(G\) is virtually abelian since 
\(\langle v_1^2, (v_1 v_2)^2, h \rangle\) is abelian with index: 
\[
[G : \langle v_1^2, (v_1 v_2)^2, h \rangle] =  
[Q : \langle v_1^2, (v_1 v_2)^2 \rangle] =
[Q : \langle v_1^2, v_1 v_2 \rangle] [\langle v_1^2, v_1 v_2 \rangle: \langle v_1^2, (v_1 v_2)^2 \rangle] \leq 2 \cdot 2 = 4.  
\]

\textbf{Case 3:} \(\epsilon_1 = \epsilon_2 = -1\). 
In this case, the presentation for \(G\) is 
\[\langle v_1, v_2, h \mid v_i h = h^{-1} v_i, v_1^2 v_2^2 = h^b\rangle.\]  
In particular, \(h\) commutes with \(v_1^2\) and \(v_1 v_2\) and
\begin{align*} 
v_1^2 \cdot v_1 v_2 \cdot v_1^{-2} & = v_1 v_1^2 v_2^2 v_2 v_2^{-2} v_1^{-2}  = v_1 h^b v_2 h^{-b}  = h^{-2b} v_1 v_2.  
\end{align*} 

We now consider \(b = 0\) and
	\(b \neq 0\) separately.

\textbf{Subcase 3.1:} \(b = 0\). In this case, 
\(\langle v_1^2, v_1 v_2, h \rangle\) is an abelian subgroup of \(G\) with index
\([G : \langle v_1^2, v_1 v_2, h \rangle]
	= [Q: \langle v_1^2, v_1 v_2 \rangle] = 2\), and so \(G\) is virtually
	abelian and hence has a decidable Diophantine problem.

\textbf{Subcase 3.2:} \(b \neq 0\). 
In this case we shall appeal to \cref{comm-undecidable} to show the Diophantine problem is undecidable. Since $[v_1^2, v_1v_2] = h^{-2b}$ and $h$ commutes with both $v_1^2$ and $v_1v_2$ in $G$, it then follows from Lemma~\ref{lem:useful} that       
\begin{align}\label{eqn:h2bij}
[v_1^{2i}, (v_1 v_2)^j] = h^{-2bij}
\end{align}
for all \(i, j \in \Z\).

	We next show that the sets \(\langle h \rangle\), 
$\langle v_1,h \rangle$, 
\(\langle v_1^2, h^2
	\rangle\) 
 and \(\langle v_1 v_2, h \rangle\) are all equationally definable,
	which will then allow us to apply \cref{comm-undecidable}. 
Since the defining relations preserve the parity of the sum of the exponent sum of $v_1$ and the exponent sum of $v_2$ we have   
\[
\langle v_1,h \rangle \cap \langle v_1v_2 ,h \rangle =
\langle v_1v_1, h \rangle \cap \langle v_1v_2 ,h \rangle = \langle h \rangle
\]
so once the  sets $\langle v_1,h \rangle$ and $\langle v_1v_2 ,h \rangle$ have been shown to be equationally definable it will follow that $\langle h \rangle $ also is.

We begin by showing \(C_G(v_1^2) =
	\langle v_1, h \rangle\). It is immediate that \(C_G(v_1^2) \supseteq
	\langle v_1, h \rangle\). Let \(h^k v_1^{2i} (v_1 v_2)^j v_1^\delta
	\in C_G(v_1^2)\). Then using the fact that $h$ commutes with any word over $\{ v_1, v_2\}$ of even length we obtain:        
	\begin{align*}
		h^k v_1^{2(i + 1)} (v_1 v_2)^j v_1^\delta 
		& = v_1^2 \cdot h^k v_1^{2i} (v_1 v_2)^j v_1^\delta 
		& & = h^k v_1^{2i} (v_1 v_2)^j v_1^\delta \cdot v_1^2 \\
		 & = h^k v_1^{2i} v_1^2 (v_1^{-2} v_1 v_2 v_1^2)^j v_1^\delta 
		& & = h^k v_1^{2(i + 1)}  (v_1 v_1^{-2} v_2 v_2^{-2} h^b)^j v_1^\delta \\
		 & = h^k v_1^{2(i + 1)} (v_1 v_1^{-2} v_2^{-2} v_2 h^b)^j v_1^\delta 
		& & = h^k v_1^{2(i + 1)} (v_1 h^{-b} v_2 h^b)^j v_1^\delta \\
		 & = h^k v_1^{2(i + 1)} (h^{2b} v_1 v_2)^j v_1^{\delta} 
		& & = h^{k + 2bj} v_1^{2(i + 1)} (v_1 v_2)^j v_1^{\delta}. 
	\end{align*}
	Thus \(j = 0\) and so \(C_G(v_1^2) = \langle v_1, h \rangle\). Let \(h^k
	v_1^{2i + \delta} \in \langle v_1, h \rangle\).  Then \((h^k v_1^{2i +
	\delta})^2 = h^{k + (-1)^\delta k} v_1^{4i + 2\delta}\).  It follows that the
	set of sqaures in \(\langle v_1, h \rangle\) is \(\langle h^2, v_1^2
	\rangle\), and so this set is equationally definable.

	We next show that \(C_G(v_1 v_2) = \langle v_1 v_2, h \rangle\). Again,
	\(C_G(v_1 v_2) \supset \langle v_1 v_2, h \rangle\) is immediate, so
	it remains to show the reverse containment. Let
	\(h^k v_1^{2i} (v_1 v_2)^j v_1^\delta \in C_G(v_1 v_2)\). Then
	the image of \(h^k v_1^{2i} (v_1 v_2)^j v_1^{\delta}\) must project
	onto an element that centralises the image of \(v_1 v_2\) in \(Q\),
	which implies that \(\delta = 0\). Furthermore,
	\begin{align*}
					h^k v_1^{2i} (v_1 v_2)^{j + 1} 
					& = v_1 v_2 \cdot h^k v_1^{2i} (v_1 v_2)^j 
					& & = h^k (v_1 v_2 v_1^2 v_2^{-1} v_1^{-1})^i (v_1 v_2)^{j + 1} \\
					& = h^k (v_1 v_2 h^b v_2^{-3} v_1^{-1})^i (v_1 v_2)^{j + 1} 
					& & = h^k (h^b v_1 v_2^{-2} v_1^{-1})^i (v_1 v_2)^{j + 1} \\
					& = h^k (h^b v_1 h^{-b} v_1^2 v_1^{-1})^i (v_1 v_2)^{j + 1} 
					& & = h^k (h^{2b} v_1^2)^i (v_1 v_2)^{j + 1} \\
					& = h^{k + 2bi} v_1^{2i} (v_1 v_2)^{j + 1}.
	\end{align*}
	Thus \(i = 0\), and so \(C_G(v_1 v_2) = \langle v_1 v_2, h \rangle\),
	and we have shown this set is equationally definable.

Now set $x=v_1^2$ and $y=v_1v_2$ and let 
$S_x = \langle v_1^2,h^2 \rangle$ which is the set of all squares of elements of $\langle v_1, h \rangle$, and let 
$S_y = \langle v_1v_2, h \rangle$.           
We proved above that $[x^i,y^j] = h^{{-2b}ij}$ and that $\langle h \rangle $, $S_x$ and $S_y$ are all equationally definable.     
Also 
$S_x / \langle h \rangle 
= \langle v_1^2,h^2 \rangle / \langle h \rangle 
= \langle v_1^2 \rangle / \langle h \rangle$
and 
$S_y / \langle h \rangle 
= \langle v_1v_2,h \rangle / \langle h \rangle 
= \langle v_1v_2 \rangle / \langle h \rangle$.
Hence all the hypotheses of \cref{comm-undecidable} hold, and so Hilbert's tenth problem is
reducible to the Diophantine problem in \(G\), and so the latter is undecidable.
\end{proof}

\begin{lem}[n1-2, (1; 2, 2)]
		\label{n12-1-2-2}
	Let \(b \in \Z\) and \(\epsilon \in \{1, -1\}\). Let \(G\) be the
	group with the presentation
	\[
		\langle v_1, q_1, q_2, h \mid v_1 h = h^\epsilon v_1, q_i h = hq_i,
		q_j^2 h = 1, q_1 q_2 v_1^2 = h^b \rangle.
	\]
	Then
    \begin{enumerate}
        \item If \(\epsilon = 1\) then \(G\) is virtually abelian
				and hence has a decidable Diophantine problem;
        \item If \(\epsilon = -1\), then the following
                are equivalent:
                \begin{enumerate}
                    \item \(b = -1\);
                    \item \(G\) is virtually abelian;
                    \item \(G\) has a decidable Diophantine problem.
                \end{enumerate}
Furthermore, in the case that the above conditions are not satisfied then Hilbert's tenth problem over $\Z$  is reducible to the Diophantine problem in $G$.   
    \end{enumerate}  
\end{lem}

\begin{proof}
	Note that the generator \(q_2 = q_1^{-1} h^b v_1^{-2}\) 
is redundant, so \(G\) admits
	the presentation
	\[
		\langle v_1, q_1, h \mid v_1 h = h^{\epsilon} v_1,
		q_1 h = h q_1, (v_1^2 h^{-b} q_1) h =
		h (v_1^2 h^{-b} q_1), q_1^2 h = 1, (v_1^2 h^{-b} q_1)^2 h^{-1} = 1
		\rangle.
	\]
	This can be further rewritten as
	\[
		\langle v_1, q_1, h \mid v_1 h = h^{\epsilon} v_1,
		q_1 h = h q_1, q_1^2 h = 1, (v_1^2 q_1)^2 = h^{2b + 1} 
		\rangle.
	\]

	Let \(Q = G / \langle h \rangle\). Then \(Q\) has the presentation
	\(\langle v_1, q_1 \mid q_1^2 = 1, (v_1^2 q_1)^2 = 1\rangle\).  We have
	from \cref{Seifert-Euler-qts} that \(Q\) is virtually
	\(\Z^2\), however we require the explicit generating set, so we show that
	\(\langle v_1^2, (v_1 q_1)^2 \rangle \leq Q\) is abelian of finite index.
	We have
	\begin{align*}
			v_1^2 (v_1 q_1)^2 & =_Q v_1 \cdot (v_1^2 q_1)^2 \cdot q_1 v_1^{-1} q_1 
							  =_Q v_1 q_1 v_1^{-1} q_1 =_Q v_1 q_1 v_1^{-1} (v_1^2 q_1)^2 q_1 =_Q
(v_1 q_1)^2 v_1^2.
	\end{align*}
Thus \(\langle v_1^2, (v_1 q_1)^2 \rangle \leq Q\) is abelian. This is a normal subgroup of $Q$ since $q_1$ has finite order in $Q$ and   
	\begin{align*}
		v_1 \cdot (v_1 q_1)^2 \cdot v_1^{-1}
		& =_Q v_1^2 q_1 v_1^2 q_1 q_1 v_1^{-1} q_1 v_1^{-1} =_Q
		(q_1 v_1^{-1})^2 =_Q (v_1 q_1)^{-2} \in \langle v_1^2, (v_1 q_1)^2
		\rangle \\
		q_1 \cdot v_1^2 \cdot q_1^{-1} 
		& =_Q v_1^{-2} (v_1^2 q_1)^2
		=_Q v_1^{-2} \in \langle v_1^2, (v_1 q_1)^2 \rangle \\
		q_1 \cdot (v_1 q_1)^2 \cdot q_1^{-1}
		& =_Q q_1 v_1^{-1} \cdot v_1^2 q_1 v_1^2 q_1 \cdot
		q_1 v_1^{-1} =_Q (q_1 v_1^{-1})^2 \in \langle v_1^2,
		(v_1 q_1)^2 \rangle.
	\end{align*}
Note that from the first line in this displayed equation it also follows that  
$v_1^{-1} \cdot (v_1 q_1)^2 \cdot v_1 = (v_1 q_1)^{-2}$. 
We have thus shown that \(\langle v_1^2, (v_1 q_1)^2 \rangle\) is normal in
	\(Q\), and hence \(Q / \langle v_1^2, (v_1 q_1)^2 \rangle\) admits the
	presentation \(\langle v_1, q_1 \mid q_1^2 = v_1^2 = (v_1^2 q_1)^2 = (v_1
	q_1)^2 = 1 \rangle\). This is isomorphic to \(\langle v_1, q_1 \mid q_1^2 =
	v_1^2 = (v_1 q_1)^2 = 1 \rangle\) which is a presentation for the Klein
	four group. In particular, \(\langle v_1^2, (v_1 q_1)^2 \rangle\) has
	finite index in \(Q\), with \(\{1, v_1, q_1, v_1 q_1\}\) being a
	transversal. Thus \(Q\) admits the normal form \(\{v_1^{2i} (v_1 q_1)^{2j}
	v_1^{\delta_1} q_1^{\delta_2} : i, j \in \Z, \delta_1, \delta_2 \in \{0,
1\}\}\). 
Since $\langle h \rangle \cong \Z$ by \cref{h-inf-order}, it follows that 
	\[
			\{h^k v_1^{2i} (v_1 q_1)^{2j} v_1^{\delta_1} q_1^{\delta_2} : 
		i, j, k \in \Z, \delta_1, \delta_2 \in \{0, 1\}\}.
	\]	
is a set of normal forms for the group \(G\). 
We next need to compute the conjugate of \((v_1 q_1)^2\) by
	\(v_1^2\) in \(G\). In $G$ we have
	\begin{align*}
		v_1 \cdot (v_1 q_1)^2 \cdot v_1^{-1}
		& = (v_1^2 q_1)^2 q_1^{-1} v_1^{-1} q_1 v_1^{-1} 
		 = h^{2b+1} h (q_1 v_1^{-1})^2 
		 = h^{2b+1} h (v_1 q_1 h)^{-2} 
		 = h^{2b+1 - \epsilon}   (v_1 q_1)^{-2}. 
\end{align*}
This then implies 
\begin{align}
		\nonumber
v_1^2 \cdot (v_1 q_1)^2 \cdot v_1^{-2}
&=
v_1 (h^{2b + 1 - \epsilon} (v_1 q_1)^{-2}) v_1^{-1} \\
\nonumber
& = h^{\epsilon(2b +  1 - \epsilon)} (v_1 (v_1 q_1)^2 v_1^{-1})^{-1} \\
\nonumber
& = h^{\epsilon(2b + 1 - \epsilon)} (h^{2b + 1 - \epsilon} (v_1 q_1)^{-2})^{-1}\\
\label{eq:1-2-2-conj}
& = h^{(\epsilon - 1)(2b + 1 - \epsilon)} (v_1 q_1)^2.
\end{align}

	\textbf{Case 1:} 
\(\epsilon = 1\) or \(b = - 1\). In this case, either \(\epsilon - 1 = 0\) or
\(\epsilon = -1\), and hence \(2b + 1 - \epsilon = 0\). Thus  
Equation~\eqref{eq:1-2-2-conj} shows that \(v_1^2\) and \((v_1 q_1)^2\) commute. Since these both
commute with \(h\), we have that \(\langle v_1^2, (v_1 q_1)^2, h \rangle\) is
an abelian subgroup of \(G\) with index 
\([G: \langle
	v_1^2, (v_1 q_1)^2, h \rangle] = [Q: \langle v_1^2, (v_1 q_1)^2 \rangle] = 4\). Thus  \(G\) is virtually abelian, and hence has a decidable Diophantine
	problem.

\textbf{Case 2:} \(\epsilon = -1\) and \(b \neq -1\).  
In this case we shall apply \cref{comm-undecidable} to show that \(G\) has an undecidable Diophantine problem. From Equation~\eqref{eq:1-2-2-conj}, it follows that in $G$ we have
\[
[v_1^{2}, (v_1 q_1)^{2}] = h^{(-2)(2b + 2)} = h^{{-4}(b+1)}. 
\]
Since $h$ commutes with $v_1^2$ and $h$ commutes with $(v_1 q_1)^{2}$ it then follows from Lemma~\ref{lem:useful} that  
\begin{align}\label{eqn:vqhij} 
[v_1^{2i}, (v_1 q_1)^{2j}] &= h^{-4ij(b + 1)}
\end{align}
for all \(i, j \in \Z\).

Next we shall prove that 
\(C_G(v_1^2) = \langle v_1, h \rangle\). Clearly,
	\(\langle v_1, h \rangle \subseteq C_G(v_1^2)\). 
	Before we attempt the other containment, we must compute \(q_1 v_1^2 q_1^{-1}\)
	in \(G\). We have
	\[
		q_1 v_1^2 q_1^{-1}
		= q_1 v_1^2 q_1h
		= v_1^{-2} (v_1^2 q_1)^2 h 
		= v_1^{-2} h^{2b + 2}.
	\]
	Returning to \(C_G(v_1^2)\), let \(h^k v_1^{2i} (v_1 q_1)^{2j}
	v_1^{\delta_1} q_1^{\delta_2} \in C_G(v_1^2)\), where \(i, j, k \in \Z\)
	and \(\delta_1, \delta_2 \in \{0, 1\}\). Then
	\begin{align*}
			h^{k} v_1^{2(i + 1)}  (v_1 q_1)^{2j} v_1^{\delta_1} q_1^{\delta_2}
		& = v_1^2 \cdot h^k v_1^{2i} (v_1 q_1)^{2j} v_1^{\delta_1} q_1^{\delta_2} 
		 = h^k v_1^{2i} (v_1 q_1)^{2j} v_1^{\delta_1} q_1^{\delta_2} \cdot v_1^2\\
		& = h^k v_1^{2i} (v_1 q_1)^{2j} v_1^{\delta_1} v_1^{2(-1)^{\delta_2}} h^{\delta_2(2b + 2)} q_1^{\delta_2}  
		 = h^k v_1^{2i} \cdot (v_1 q_1)^{2j} v_1^{2(-1)^{\delta_2}} \cdot v_1^{\delta_1} h^{\delta_2(2b + 2)} q_1^{\delta_2} \\ 
		& = h^k v_1^{2(i + (-1)^{\delta_2})} (v_1 q_1)^{2j} h^{-4(b + 1)(-1)^{\delta_2}j}
		v_1^{\delta_1}
		h^{\delta_2(2b + 2)} q_1^{\delta_2} \quad \mbox{(by Equation~\ref{eqn:vqhij})} \\
		& = h^{k -4(b + 1)(-1)^{\delta_2}j + (-1)^{\delta_1} \delta_2(2b + 2)} 
		v_1^{2(i + (-1)^{\delta_2})} (v_1 q_1)^{2j} v_1^{\delta_1}
		q_1^{\delta_2}. 
	\end{align*}
	It follows that \(i + (-1)^{\delta_2} = i + 1\), and so \(\delta_2 = 0\).
	In addition, it follows that \(k -4(-1)^{\delta_2}j + (-1)^{\delta_1}
	\delta_2(2b + 2) = k\). Plugging in \(\delta_2 = 0\) gives \(-4(b + 1)j =
	0\), and so \(j = 0\).  We can conclude that \(C_G(v_1^2) \subseteq \langle
	v_1, h \rangle\), and so these sets are equal.  

	We next show that \(C_G((v_1 q_1)^2) = \langle v_1 q_1, h \rangle\). It is
	immediate that \(\langle v_1 q_1, h \rangle \subseteq C_G((v_1 q_1)^2)\).
	Let \(g  =  h^k v_1^{2i} (v_1 q_1)^{2j} v_1^{\delta_1} q_1^{\delta_2}
	\in C_G((v_1q_1)^2)\), where \(i, j, k \in \Z\) and \(\delta_1, \delta_2
	\in \{0, 1\}\).  Then the projection \(\bar{g}  =  v_1^{2i} (v_1
	q_1)^{2j} v_1^{\delta_1} q_1^{\delta_2} \) of \(g\) in \(Q\) lies in
	\(C_Q((v_1 q_1)^2)\). Thus, using our earlier computations that \(v_1 (v_1
	q_1)^2 v_1^{-1} =_Q  (v_1 q_1)^{-2}\) and
	\(q_1 (v_1 q_1)^2 q_1^{-1} =_Q (v_1 q_1)^{-2}\), we have
	\begin{align*}
			v_1^{2i} (v_1 q_1)^{2(j + 1)} v_1^{\delta_1} q_1^{\delta_2} & =_Q
			(v_1 q_1)^2 \cdot v_1^{2i} (v_1 q_1)^{2j} v_1^{\delta_1}
			q_1^{\delta_2} \\ & =_Q v_1^{2i} (v_1 q_1)^{2j} v_1^{\delta_1}
			q_1^{\delta_2} (v_1 q_1)^2 & =_Q v_1^{2i} (v_1 q_1)^{2j +
			2(-1)^{\delta_1 + \delta_2}} v_1^{\delta_1} q_1^{\delta_2}.
	\end{align*}
	We can conclude that \(\delta_1 + \delta_2 \equiv 0 \mod 2\). Returning to
	\(G\), we can thus write \(g = h^k v_1^{2i} (v_1 q_1)^{2j} (v_1
	q_1)^{\delta}\), where \(\delta \in \{0, 1\}\). We have
	\begin{align*}
		h^k v_1^{2i} (v_1 q_1)^{2(j + 1)} (v_1 q_1)^\delta 
		& = (v_1 q_1)^2 \cdot h^k v_1^{2i} (v_1 q_1)^{2j} (v_1 q_1)^{\delta} 
		  = h^{k - 4i(b + 1)} v_1^{2i} (v_1 q_1)^{2(j + 1)} (v_1 q_1)^\delta.
	\end{align*}
	Thus \(i = 0\) and \(g = h^k (v_1 q_1)^{2j} (v_1 q_1)^\delta \in \langle h,
	v_1 q_1 \rangle\). We can conclude that \(C_G((v_1 q_1)^2) = \langle v_1
	q_1, h \rangle\), and thus this set is equationally definable.

Now set $x=v_1^2$ and $y=(v_1q_1)^2$ and let $S_x = \langle v_1^2,h^2 \rangle$ which is the set of all squares of elements of $\langle v_1, h \rangle$, and let $S_y = \langle (v_1q_1)^2, h^2 \rangle$
which is the set of all squares of elements of $\langle v_1q_1, h \rangle$. We proved above that  
$\langle v_1, h \rangle$ and $\langle v_1q_1, h \rangle$ are both equationally definable and so it follows that the sets $S_x$ and $S_y$ are also both equationally definable.  
We proved above that $[x^i,y^j] = h^{-4ij(b + 1)}$ and we have  \(\langle h \rangle = C_G(v_1^2) \cap C_G((v_1 q_1)^2)\), and thus $\langle h \rangle$ is equationally definable.
Also $S_x / \langle h \rangle = \langle v_1^2,h^2 \rangle / \langle h \rangle = \langle v_1^2 \rangle / \langle h \rangle$ and $S_y / \langle h \rangle = \langle (v_1q_1)^2, h^2 \rangle / \langle h \rangle = \langle (v_1q_1)^2 \rangle / \langle h \rangle$. Thus all the hypotheses of \cref{comm-undecidable} hold, and so 
Hilbert's tenth problem is
reducible to the Diophantine problem in \(G\), and so the latter is undecidable.
\end{proof}

\section{Seifert 3-manifold groups with negative 
characteristic and \\ statement of main result}
In this section we consider the Seifert 3-manifold groups $\pi_1(M)$ where  $M$ be a Seifert fibered manifold with base orbifold $B$ satisfies $\chi(B) < 0$.  If $G$ is such a group then by \cref{Seifert-Euler-qts} the quotient $G / \langle h \rangle$ is hyperbolic. In \cite{Liang2014} Liang proved that central extensions of hyperbolic groups have a decidable Diophantine problem. Hence if $G / \langle h \rangle$ is hyperbolic and $h$ is central then Liang's result implies that $G$ has a decidable Diophantine problem. However, in general in the case  $\chi(B) < 0$ the element $h$ need not be central. The following result shows that in the cases that $h$ is not central we can still find an index $2$ subgroup of the Seifert 3-manifold group in which the Diophantine problem is decidable.

\begin{proposition}\label{proposition:hyper} Let $G = \pi_1(M)$ where $M$ be a Seifert fibered manifold with base orbifold $B$ satisfies $\chi(B) < 0$. Then either $G$ has a decidable Diophantine problem or $G$ contains an index $2$ subgroup $H$ such that $H$ has a decidable Diophantine problem. In more detail:
\begin{enumerate} 
\item if $h$ is central then $G$ has a decidable Diophantine problem;    
\item if $h$ is not central then $G$ contains an index $2$ subgroup $H$ such that $H$ has a decidable Diophantine problem.    
  \end{enumerate}
  \end{proposition}
\begin{proof} 
		Let $G = \pi_1(M)$ where $M$ be a Seifert fibered manifold with
		base orbifold $B$ satisfies $\chi(B) < 0$. If $h$ is central then by
		\cref{Seifert-Euler-qts} the group $G$ is a central extension of a
		hyperbolic group and thus has a decidable Diophantine problem by a theorem of Liang
		\cite{Liang2014}. Now suppose that $h$ is not central. In this
		case we shall identify an index two subgroup $H \leq G$ such that $H$
		is a central extension of a hyperbolic group, and thus $H$ has
		decidable Diophantine problem again by Liang's theorem \cite{Liang2014}. We
		define the subgroup $H$ of $G$ as follows. Set  
\[I = \{ i \in \{1, \ldots, g \} : v_i h = hv_i \}. \]
Note that if $B$ is orientable then it follows from the definitions that $u_ih=hu_i$ if and only if $i \in I$. Since $h$ is not central it follows that $I$ is non-empty. Now set
\[
\Delta = \{ u_i: i \in I \} \cup \{ v_i : i \in I \}.
\]
By inspection of the defining relations we see that if $\alpha$ and $\beta$ are any two words over the generating set with $\alpha =_G \beta$ then for all $\delta \in \Delta$ we have   
\[\mathrm{expsum}_{\delta}(\alpha) \equiv  \mathrm{expsum}_{\delta}(\beta) \pmod{2} \]
where $\mathrm{expsum}_{\delta}(\alpha)$ denotes the exponent sum of the letter $\delta$ in the word $\alpha$. 
(In fact this equality holds for any letter $u_i$ or $v_j$ and does not depend on the letter coming from the set $\Delta$.) It follows that for any element $h \in G$ the number       
\[\sum_{\delta \in \Delta} \mathrm{expsum}_{\delta}(h) \pmod{2} \]
is well-defined and so we can define  
\[H = \{ k \in G: \mbox{$ \sum_{\delta \in \Delta} \mathrm{expsum}_{\delta}(k) $ is even} \}. \]
Clearly $H$ is a subgroup of $G$. Furthermore this subgroup clearly has index $2$ in $G$ where the other coset of $H$ in $G$ is the set
\[\{ l \in G: \mbox{$ \sum_{\delta \in \Delta} \mathrm{expsum}_{\delta}(k) $ is odd} \}. \]
To complete the proof of the proposition it will suffice to prove $H$ is a central extension of a hyperbolic group. To see this note that for any $k \in H$ we have $hk=kh$ since $h$ commutes with $q$ and the definition of $H$ means that as we pass the letter $h$ through any word representing $k$ the letter $h$ will be flipped between $h$ and $h^{-1}$ an even number of times. Hence $h$ is central in $H$. Furthermore we have
\[[G / \langle h \rangle : H / \langle h \rangle] = [G:H] = 2 \]
So $H / \langle h \rangle$ is an index 2 subgroup of the hyperbolic group $G /
\langle h \rangle$ and thus $H / \langle h \rangle$ is also a hyperbolic group.
It follows that $H$ is a central extension of the hyperbolic group $H / \langle
h \rangle$ and thus $H$ has a decidable Diophantine problem by Liang's theorem
\cite{Liang2014}.   
  \end{proof}

We conjecture that all Seifert 3-manifold groups with 
negative orbifold Euler characteristic have decidable Diophantine problem.

We are now in a position to state and prove the first main result of this paper which has Theorem~A as an immediate corollary.   
\begin{theorem}
			\label{Seifert-DP}
Let $M$ be a Seifert fibered manifold with base orbifold $B$ and set $G=\pi_1(M)$. 	
Then, using the notation of \cref{Seifert-pres}, we have the following: 
				\begin{enumerate}
								\item If \(\chi(B) < 0\), then 
\begin{enumerate} 
\item \(G\) has a decidable Diophantine problem in the cases: (o1) and (n1);
\item \(G\) has an index $2$ subgroup $H$ with  decidable Diophantine problem in all the other cases that $\chi(B) < 0$.  
  \end{enumerate}
								\item If \(\chi(B) = 0\), then:
												\begin{enumerate}
													\item If \(g = 0\) then \(G\) has an undecidable Diophantine
												problem, except in the following cases:
									\begin{enumerate}
													\item \(G\) is in (0; 2, 2, 2, 2) and \(b = -1\);
													\item \(G\) is in (0; 3, 3, 3) and \(\beta_1 + \beta_2 + \beta_3 + 3b = 0\);
													\item \(G\) is in (0; 2, 4, 4) and \(2 + \beta_2 + \beta_3 + 4b = 0\);
													\item \(G\) is in (0; 2, 3, 6) and \(3 + 2 \beta_2 + \beta_3 + 6b + 3 = 0\).
									\end{enumerate}
									In each of these cases \(G\) has a 
decidable first-order theory with ETD0L solutions. In particular $G$ has a decidable Diophantine problem with EDT0L solutions.
					\item If \(g = 1\) and \(G\) has type (o1) then \(G\) has a decidable
									Diophantine problem if and only if \(b = 0\), in which case
$G$ has a decidable first-order theory with ETD0L solutions;
	\item If \(g = 1\) and \(G\) has type (n2), then \(G\) has
							a decidable Diophantine problem if and only if
							\(b = -1\), in which case 
$G$ has a decidable first-order theory with ETD0L solutions.
\item If \(g = 2\) and \(G\) has type (n2), then \(G\) has a
									decidable Diophantine problem if and only if \(b = 0\),
									in which case 
$G$ has a decidable first-order theory with ETD0L solutions.
					\item 
In all of the following cases: 
\begin{enumerate} 
\item 
\(g = 1\) and \(G\) has type (o2); or
\item
\(g = 1\) and \(G\) has type (n1); or  
\item
\(g = 2\) and \(G\) has type (n1) or (n3);  
\end{enumerate}
the group \(G\) has a decidable first-order theory with ETD0L solutions. In particular $G$ has a decidable Diophantine problem with EDT0L solutions.
												\end{enumerate}
								\item If \(\chi(B) > 0\), then \(G\) has a 
decidable first-order theory with ETD0L solutions. In particular $G$ has a decidable Diophantine problem with EDT0L solutions.
				\end{enumerate}
  Furthermore, in all cases when the Diophantine problem is undecidable  
Hilbert's tenth problem over $\Z$  is reducible to the Diophantine problem in $G$.   
\end{theorem}
	\begin{proof}
		\textbf{Case 1:} \(\chi(B) = 0\). By \cref{Seifert-Euler-qts}, this case occurs
		precisely when: 
\begin{multicols}{2}
	\begin{itemize}
		\item Type (o1), (0; 2, 2, 2, 2);
		\item Type (o1), (0; 3, 3, 3);
		\item Type (o1), (0; 2, 4, 4);
		\item Type (o1), (0; 2, 3, 6); 
		\item Type (o1-2), (1;);
		\item Type (n1-2), (1; 2, 2);
		\item Type (n1-3), (2;).
	\end{itemize}
  \end{multicols}
		If \(g = 0\), then \cref{o1-0-2-2-2-2},
		\cref{o1-0-3-3-3}, \cref{o1-0-2-4-4} and \cref{o1-0-2-3-6}, together with
		the fact that virtually abelian groups have the stated properties
by Remark~\ref{rmk:virtuallyAbelianEDT0L}l	
give the result. If \(g = 1\), then
		\cref{o1-1}, \cref{o2-1} and \cref{n12-1-2-2} 
		give the result. If \(g = 2\), then \cref{n123-2} gives the result.

		\textbf{Case 2:} \(\chi(B) > 0\).

		From \cref{Seifert-Euler-qts} if \(\chi(B) >
		0\), then \(G / \langle h \rangle\) is finite, in which case \(G\) is
		finite-by-[f.g. abelian]. It follows that \(G\) is quasi-isometric to an
		abelian group, and hence by e.g. \cite[Section~50]{delaHarpe} is virtually abelian.
		This deals with this case since virtually abelian groups have the desired properties by Remark~\ref{rmk:virtuallyAbelianEDT0L}.

		\textbf{Case 3:} \(\chi(B) < 0\).

This case is covered by Proposition~\ref{proposition:hyper}.		
	\end{proof}

\section{The single equation problem}
All 3-manifold groups are known to have decidable conjugacy problem. Above we have seen that there are 3-manifold groups with undecidable Diophantine problem. Since conjugacy is defined using a single equation it is natural to ask whether the single equation problem might be decidable for 3-manifold groups. In this section we consider that problem for Seifert manifold groups. We shall prove: 

\begin{theorem}\label{thm:singleEqn} 
Let $M$ be a Seifert fibered manifold with base orbifold $B$.
\begin{enumerate} 
\item If $\chi(B) \geq  0$ then the Seifert manifold group $G=\pi_1(M)$ has a decidable single equation problem. 
\item If $\chi(B) <  0$ then either $\pi_1(M)$ has a decidable single equation problem or it has an index 2 subgroup that does.   
  \end{enumerate}
  \end{theorem}
We conjecture that all Seifert 3-manifold groups have decidable single equation problem. We also ask whether this extends to all 3-manifold groups: 
 
\begin{qu}
Is the single equation problem solvable for the fundamental group of any $3$-manifold? 
  \end{qu}
 
We can further ask if the set of solutions to a single equation in a
\(3\)-manifold group can be expressed using an EDT0L language. This has been
shown to be true for virtually abelian groups \cite{VAEP}, but other Seifert
\(3\)-manifold groups have not been covered. There has been some work on the
nilpotent case, with single equations with one variable in the Heisenberg group
having been shown to have EDT0L solutions \cite{LevineHeisEDT0L}. 

\begin{qu}
Is the set of solutions to a single equation in the fundamental group of any $3$-manifold expressible as an EDT0L language? 
  \end{qu}

	We begin with a statement about Seifert 3-manifold groups where the base
	orbifold has characteristic zero.

	\begin{proposition}
		\label{euclidean-Seifert-virt-nilp}
		Let \(M\) be a Seifert fibered manifold with base orbifold \(B\), such
		that \(\chi(B) = 0\). Then \(\pi_1(M)\) admits a finite index subgroup
		\(N\) such that \(N\) is class \(2\) nilpotent and \([N, N]\) is
		cyclic.
	\end{proposition}

	\begin{proof}
		We begin by showing that \(G  =  \pi(M)\) admits a finite index
		subgroup \(E\) that is a central extension \(1 \to \Z \to E \to R \to
		1\), where \(Q\) is virtually \(\Z^2\). Using the notation of
		\cref{Seifert-pres}, we have from \cref{h-inf-order} that \(\langle h
		\rangle \cong \Z\), and from \cref{Seifert-Euler-qts} that \(G /
		\langle h \rangle\) is virtually \(\Z^2\).

		Set \(E = C_G(h)\). Clearly, \(\langle h \rangle \leq E\) is
		isomorphic to \(\Z\) and central in \(E\). Moreover, if
		\(E\) is finite index in \(G\), then \(E / \langle h \rangle\)
		is finite index in \(G / \langle h \rangle\), which is virtually
		\(\Z^2\). In order to show that \(E\) is indeed a central
		extension \(1 \to \Z \to E \to R \to 1\), with \(R\)
		virtually \(\Z^2\), it thus suffices to show that \(E\)
		has finite index in \(G\). We have that \(E = C_G(h)\) is
		the kernel of the action of \(G\) on \(\langle h \rangle\)
		by conjugation. The induced homomorphism of this action
		goes from \(G\) to \(\Aut(\Z) \cong C_2\), and thus its kernel
		(namely \(E\)) has index at most \(2\), as required.

		We next show that \(E\) (and hence \(G\)) admits a finite index
		subgroup \(N\) such that \(N\) is a central extension \(1 \to \Z \to N
		\to \Z^2 \to 1\). Since \(R\) is virtually
		\(\Z^2\), it admits a finite index subgroup \(A\), with \(A \cong
		\Z^2\). Let \(\phi \colon E \to R\) be the epimorphism from the
		short-exact sequence, and define \(N = \{g \in E : (g)\phi\in
		A\}\). Since \(\phi\) is a homomorphism, \(N\) is indeed a subgroup,
		and since \(\langle h \rangle = \ker \phi\), we have that
		\(\langle h \rangle \leq N\). Then 
		\([E:N] = [E / \langle h \rangle : N / \langle h \rangle] = [R: A] <
		\infty\). Thus \(N\) is indeed a central extension \(1 \to \Z \to
		N \to \Z^2 \to 1\).

		At this point, it is not difficult to deduce that \(N\) is class \(2\)
		nilpotent, however we will show this completely, since we need to
		compute \([N, N]\) to show it is virtually cyclic. We have that
		\(N\) is generated by \(\{h, a, b\}\), where \(a\) and \(b\) are
		lifts of the two generators of \(\Z^2\) to \(N\). Moreover, every
		element of \(N\) can be expressed in the form \(a^i b^j h^k\),
		since \(h\) is central, and \([a, b]\) lies in the kernel of
		the homomorphism from \(N\) to \(\Z^2\), which is 
		\(\langle h \rangle\). Thus commutators must lie in \(\langle
		h \rangle\), and so \([N, N] \leq \langle h \rangle\), and hence
		must be cyclic. Since \(\langle h \rangle\) is central, we have
		shown that \([N, N]\) is, and hence \(N\) is class \(2\) nilpotent.
	\end{proof}

\begin{proof}[Proof of \cref{thm:singleEqn}]
 	Let \(B\) be the base orbifold of \(M\). If \(\chi(B) \neq 0\), then
		the result follows from \cref{Seifert-DP}. 
  If \(\chi(B) = 0\), then
		\cref{euclidean-Seifert-virt-nilp} tells us that \(\pi_1(M)\) is
		virtually a class \(2\) nilpotent group with a cyclic commutator
		subgroup, and thus has a decidable single equation problem, by
		\cite{Levine22}*{Theorem 1.1}.
\end{proof}

\section{Other $3$-manifold groups}\label{sec:other}
We conclude the paper by making some observations about what is currently known about decidability of the Diophantine problem and first-order theory for other $3$-manifold groups. 

\subsection{Fundamental groups of manifolds modeled on 3-dimensional Sol geometry}

As discussed in the introduction above Seifert fibered spaces account for all compact oriented manifolds in six of the eight Thurston geometries; see e.g. \cite[Theorem 1.8.1]{AschenbrennerBook2015}. One of the remaining two Thurston geometries is the so-called 3-dimensional Sol geometry. We do not know whether such groups have decidable Diophantine problem, but they do have an undecidable first-order theory as we observe in the next result. 

\begin{proposition}\label{prop_sol} 
Let $M$ be a manifold modeled on 3-dimensional Sol geometry.  Then $\pi_1(M)$ has undecidable first-order theory.   
  \end{proposition}
\begin{proof} 
It is known (see e.g. \cite[Table~1.1]{AschenbrennerBook2015}) in this case   
$\pi_1(M)$ is a solvable but not virtually nilpotent group. Hence 
$\pi_1(M)$  is solvable but not virtually abelian. 
But results of Er\v{s}ov, Romanovski\u{i} and Noskov 
\cite{ershov1972elementary, romanovskiui1981elementary,noskov1984elementary} 
show that the first-order theory of any finitely generated non-virtually abelian solvable group is undecidable. Hence in particular $\pi_1(M)$ has 
undecidable first-order theory. 
\end{proof}
There are other 3-manifold groups that are solvable but not virtually abelian, and so the proof of the above proposition would also apply to those too. Specifically in the paper \cite{EvansSolvable1972} the authors give a classification of
those solvable groups which can occur as the fundamental group of a compact
3-manifold. Their results therefore thus give several other infinite families of $3$-manifold groups with undecidable first-order theory.   
The following result is a corollary of \cite[Corollary~4.6]{EvansSolvable1972} and \cite{ershov1972elementary, romanovskiui1981elementary,noskov1984elementary}

\begin{cor}\label{corol_evans} 
Let $G$ be a group from one of the following families 
\begin{enumerate} 
\item An extension 
$1 \rightarrow A \rightarrow G \rightarrow \Z \rightarrow 1$  
where $A$ is either $\Z \times \Z$ or the fundamental group $\mathcal{K}$ of the Klein bottle; or     
\item A group defined by a presentation of the form 
\[ \langle a,b,x,y \mid bab^{-1} = a^{-1}, \; yxy^{-1} = x^{-1}, \; a = x^py^{2q}, \; b^2 = x^ry^{2s} \rangle \]
where $p,q,r,s$ are integers satisfying $ps - rq = \pm 1$. 
\end{enumerate} 
Then $G$ is a $3$-manifold group with undecidable first-order theory unless $G$ is virtually abelian. 
\end{cor}
We currently do not know whether the previous corollary also holds true with first-order theory replaced by the Diophantine problem. For part (1) this relates to the problem of understanding the Diophantine problem for abelian-by-cyclic groups which has received attention in the literature e.g. in \cite{kharlampovich2020diophantine, Dong2025}.  
This relates to the following natural question arising from the work in this paper: 
\begin{qu} 
Is there a 3-manifold group with a decidable Diophantine problem but undecidable first-order
theory?  
  \end{qu}

\subsection{Constructing further examples}

The following easy proposition and corollary are included just to illustrate how the Seifert manifolds  with undecidable Diophantine problem described in this paper can be used to build other 3-manifold groups whose fundamental groups also have undecidable Diophantine problem.

\begin{proposition} 
Let $G$ be a Seifert $3$-manifold group with non-trivial centre and let $H$ be any group. If $G$ has undecidable Diophantine problem then so does $G \ast H$.  
  \end{proposition}
\begin{proof} 
We have $Z(G) = \langle h \rangle \neq 1$, and  
\[
C_{G \ast H}(h) = C_G(h) = G.
\] 
Thus $G$ is equationally definable in $G \ast H$, completing the proof.  
  \end{proof}
\begin{cor} 
There are 3-manifold groups with undecidable Diophantine problem that are 
not virtually nilpotent and are not Seifert $3$-manifold groups. 
  \end{cor}
\begin{proof} 
Take $H$ to be any non-trivial $3$-manifold group in the previous proposition. 
Since the class of $3$-manifold groups is closed under free products the group $G \ast H$ is a $3$-manifold group.   
Then $G \ast H$ contains $F_2$ so is not virtually nilpotent. 
The group $G \ast H$ cannot be a Seifert fibered manifold group since if it were since the group $G \ast H$ is not finite it would follow from \cite{BoyerRolfsenWiest}*{Proposition 4.1(1)} that 
that $G \ast H$ has an infinite cyclic normal subgroup
(corresponding to the subgroup generated by $h$ in the statement of Theorem~\ref{Seifert-pres}) 
which is a contradiction since $G$ and $H$ are both non-trivial groups.     
  \end{proof}

The results in this paper give an up to index 2 classification of Seifert $3$-manifold groups with decidable Diophantine problem and first-order theory. It would be interesting to see whether these results could be extended to the graph manifolds in the sense of e.g. \cite{neumann2007graph}.   
Also note that for any forest $F$ the right-angled Artin group $A(F)$ is a $3$-manifold group \cite{droms1987graph}. Since right-angled Artin groups all have decidable Diophantine problem all these $3$-manifold groups also do. 

\bibliography{references}
\bibliographystyle{abbrv}
\end{document}